\documentclass[11pt]{amsart}

\usepackage{amssymb,amsthm, amssymb, amscd, amsfonts,mathrsfs}
\usepackage{color}
\usepackage[active]{srcltx} 
 \usepackage{hyperref} 
 
\title{On Hopf algebras whose coradical is a cocentral abelian cleft extension}

\author{G.~A.~Garc\' ia}
\address{CMaLP, Departamento de Matem\'atica, Facultad de Ciencias Exactas,
Universidad Nacional de La Plata. CIC-CONICET. (1900) La Plata, Argentina.}
\email{ggarcia@mate.unlp.edu.ar}

\author{M.~Mastnak}
\address{Department of Mathematics and Computing Science, Saint Mary's University, 923 Robie St, Halifax, Nova Scotia, Canada B3N 1Z9}
\email{mmastnak@cs.smu.ca}

\thanks{\noindent 2020 \emph{MSC:}\,  16T05  ---
}

\def\M{\mathcal{M}}

\def\Oc{\mathcal{O}}

\def\B{\mathcal{B}}
\def\J{\mathcal{J}}
\def\D{\mathcal{D}}
\def\YD #1#2 { _{#1}^{#2}\mathcal{Y}\mathcal{D}}
\newcommand{\ydh}{{}^H_H\mathcal{YD}}
\def\Zn{\mathbb{Z}_n}
\def\Zm{\mathbb{Z}_m}
\def\ot{\otimes}
\def\kk{\Bbbk}
\def\wx{\widehat{x}}
\def\de{\delta}
\def\ep{\epsilon}
\def\act{\cdot}
\def\dact{*}
\def\De{\Delta}
\def\bv#1#2{v_{#1}^{(#2)}}
\def\by#1#2{y_{#1}^{(#2)}}
\def\bu#1#2{u_{#1}^{(#2)}}

\def\Alg{\operatorname{Alg}}
\def\ord{\operatorname{ord}}
\def\id{\operatorname{id}}
\def\ad{\operatorname{ad}}
\def\sgn{\operatorname{sgn}}
\def\Map{\operatorname{Map}}
\def\End{\operatorname{End}}
\def\Ker{\operatorname{Ker}}
\def\Img{\operatorname{Im}}
\def\Fun{\operatorname{Fun}}
\def\GL{\operatorname{GL}}

\def\CC{\mathbb{C}}
\def\N{\mathbb{N}}
\def\B{\mathfrak{B}}
\def\ot{\otimes}

\def\Stab{\operatorname{Stab}}

\def\hit{\rightharpoonup}

\newtheorem{lemma}{Lemma}[subsection]
\newtheorem{theorem}[lemma]{Theorem}
\newtheorem{corollary}[lemma]{Corollary}
\newtheorem{proposition}[lemma]{Proposition}
\newtheorem{remark}[lemma]{Remark}

\newtheorem{definition}[lemma]{Definition}

\begin{document}

\maketitle

\begin{abstract}
This paper is a first step toward the full description of a family of 
Hopf algebras whose coradical is isomorphic to
a semisimple Hopf algebra $K_{n}$, $n$ an odd positive integer, obtained by a cocentral abelian cleft extension. We describe the simple Yetter-Drinfeld modules, compute 
the fusion rules and determine the finite-dimensional Nichols
algebras for some of them. In particular, we give the description of the 
finite-dimensional Nichols 
algebras over simple modules over $K_{3}$. This includes a family of 
$12$-dimensional Nichols algebras $\{\B_{\xi}\}$ depending on  
$3$rd roots of unity. Here, $\B_{1}$ is isomorphic 
to the well-known Fomin-Kirillov
algebra, and  $\B_{\xi} \simeq \B_{\xi^{2}} $ as graded algebras 
but  $\B_1$ is not isomorphic to $\B_{\xi} $ as algebra for $\xi\neq 1$. 
As a byproduct 
we obtain new Hopf algebras of dimension 216.  
\end{abstract}

\section{Introduction}
The question of classifying Hopf 
algebras of finite (Gelfand-Kirillov) dimension has been a challenging problem since the beginning of the theory
in the late 60's and beginning of the 70's. Since then, there have been only a handful of general results that help to determine 
the structure of a Hopf algebra. Among them one may cite the Kac-Zhu Theorem \cite{Z} that states that a Hopf 
algebra of prime dimension is isomorphic to a group algebra, the Nichols-Zoeller \cite{NZ}
theorem that claims that 
a finite-dimensional Hopf algebra is free over any Hopf subalgebra, or the classification of 
(almost all) finite dimensional pointed Hopf algebras with abelian coradical \cite{AS}. The key ingredient of 
this last result is the introduction of a general method to construct and classify Hopf algebras whose
coradical is a Hopf subalgebra. This method is known as the \textit{Lifting Method} and it is particularly useful to classify
finite (Gelfand-Kirillov) dimensional pointed Hopf algebras, where the coradical is a group algebra, see for instance \cite{AS}, 
\cite{AAH}, \cite{AHS}, \cite{FG}, \cite{GGI}
and \cite{GIV}, to name a few. This method was later generalized in \cite{AC}; here the coradical is replaced by the Hopf 
subalgebra generated by it. Using this, new families
of Hopf algebras where found, see for instance \cite{AA}, \cite{GJG}, \cite{BGGM}, \cite{X1}, \cite{XH}.

In the last years, the appearance of full classification results has been sparse. 
One of the reasons may lay in the lack of examples with different properties, 
as one needs to know all possible examples to have a complete set of Hopf algebras up to isomorphism. On the other hand, descriptions 
of different families of Hopf algebras can be found in the literature.
For example, those that are non-pointed
but satisfy the \textit{Chevalley Property}, i.e. the coradicals are Hopf subalgebras, see for example \cite{AGM}, \cite{AV}, \cite{Shi}, \cite{Z1, Z2,Z3}.

{
With the aim of understaning non-pointed and non-copointed Hopf algebras with the Chevalley Property,
}
we begin in this paper the study of Hopf algebras whose coradical is a 
semisimple Hopf algebra $K_{n} := \Bbbk^{\Zn\times \Zn} \rtimes_{\beta} \Bbbk \mathbb{Z}_{2}$ given by a double crossed product; 
here $n\in \N$ is odd and bigger than one. It can also be described
as an abelian extension $\Bbbk \to \Bbbk \mathbb{Z}_{n}\to K_{n} \to 
\Bbbk\mathbb{D}_{n} \to \Bbbk$, where $\mathbb{D}_{n}$ is the dihedral
group of order $2n$. 
{
Despite the fact that the algebras $K_{n}$ admit an explicit and rather clear presentation,
they are non-trivial enough to produce new families of examples of finite-dimensional Hopf algebras with the Chevalley Property 
through the process of bosonization and lifting of Nichols algebras in the category $_{K_{n}}^{K_{n}}\mathcal{YD}$ of Yetter-Drinfeld modules over $K_{n}$.
The most interesting examples are the ones where the generators of the Nichols algebras are not homogeneous with respect to 
a group-like element in $K_{n}$, i.e. the realization of the braided vector space is not principal.
}

As part of the lifting method, one needs to understand the category 
$_{K_{n}}^{K_{n}}\mathcal{YD}$. 
As a first step, we describe the simple 
objects and the fusion rules of this semisimple category, 
see Theorem \ref{thm:class-simple-objects} and 
Subsection \ref{subsec:fusion-rules}.
In order to determine all simple objects we use two different approaches. 
For low dimensional objects, say dimension one or two, we adapt the method of ``little groups'' of Wigner and Mackey,
see \cite[Subsection 8.2]{Serre}. For the remaining objects, we look 
at simple subcomodules of $K_{n}$, that is, we apply a general method provided by Radford \cite{R2}.
Besides the 
importance for our goal, the result is interesting on its own right 
as we present the corresponding fusion ring explicitly, see 
Theorem \ref{thm:fusion-ring}.
Another step of the lifting method is the determination of the 
finite-dimensional Nichols algebras. We describe some of them 
in Section \ref{sec:Nichols-alg}. 
There are families of Yetter-Drinfeld modules that consist 
of braided vector spaces of diagonal type, thus their Nichols algebras 
are determined by the work of Heckenberger \cite{He} and Angiono \cite{An}.
On the other hand, some 
Yetter-Drinfeld modules turn out to be braided vector spaces 
of rack type with non-principal realization. Moreover, these are isomorphic 
to the braided vector spaces associated with 
the dihedral rack and a constant cocycle, i.e. a conjugacy 
class of an involution in the dihedral group $\mathbb{D}_{n}$ and a
one-dimensional representation. In particular, for $n=3$ 
a family of 
$12$-dimensional Nichols algebras $\{\B_{\xi}\}$ depending on 
$3$rd roots of unity appear. The algebra $\B_{1}$ is isomorphic 
to the well-known Fomin-Kirillov
algebra, $\B_{\xi} $ and $ \B_{\xi^{2}} $ are isomorphic as graded algebras 
but $\B_1$ is not isomorphic to $\B_{\xi} $ as algebra for $\xi\neq 1$, see Theorem
\ref{thm:Bi-not-iso-B1}. 
We end the paper
with the presentation of the finite-dimensional Nichols algebras 
over simple modules when $n=3$.
As a consequence, we obtain new Hopf algebras of dimension $216$ by the process of bosonization.

In future work we intend to describe all finite-dimensional Nichols 
algebras of semisimple Yetter-Drinfeld modules together with their 
liftings in order to obtain all Hopf algebras whose coradical
is isomorphic to $K_{n}$.

The article is organized as follows. In Section \ref{sec:prelim}
we include definitions and basic facts that are needed along the paper; 
in particular, we recall the definition of Yetter-Drinfeld modules
and Nichols algebras. In Section \ref{sec:Kn} we describe explicitly
the family of Hopf algebras $K_{n}$, 
whereas in Section \ref{sec:simple-YD-modules} we determine all 
simple Yetter-Drinfeld modules over $K_{n}$. 
As the category is semisimple, because $K_{n}$ is a semisimple algebra,
this is enough to describe all objects. In Section \ref{sec:fusion-ring}
we compute the fusion rules of $_{K_{n}}^{K_{n}}\mathcal{YD}$
and in Section \ref{sec:Nichols-alg} we determine the Nichols algebras associated with some 
modules in ${}_{K_{n}}^{K_{n}}\mathcal{YD}$.

\section*{Acknowledgements}
This work was supported, in part, by CONICET, ANPCyT (PICT 2018-00858), Secyt-UNLP (Argentina), and NSERC (Canada).

\section{Preliminaries}\label{sec:prelim}

Let $n\in \N$ and let $\Bbbk$ be a field containing a primitive $n$-th root of unity. 
We assume also that the characteristic of $\Bbbk$ is either 
zero or does not divide $2n$.  
All vector spaces are considered over $\Bbbk$ and
$\ot = \ot_{\Bbbk}$.
Given a group $G$, we denote by $\widehat{G}$ its character group. 
For $m\in\N$, we denote by $\Zm$ the ring of integers module $m$. 
We work with Hopf algebras $H$ over $\Bbbk$; as usual, we write 
$\Delta$, $S$ and $\varepsilon$ to denote the comultiplication,
the antipode and the counit, respectively. Also, the comultiplication
and the comodule structures
are written using Sweedler's notation, \textit{i.e.}
$\Delta(h)=h_{(1)}\ot h_{(2)}$ for all $h\in H$ and 
$\delta(v)= v_{(-1)}\ot v_{(0)}$ 
for a left $H$-comodule $(V,\delta)$  and $v\in V$.
{ The (left) adjoint action
of a Hopf algebra $H$ on itself is
denoted by $h\rightharpoonup x = h_{(1)}xS(h_{(2)})$ for all $h,x\in H$.}
We refer to \cite{R} 
for Hopf algebras and \cite{A}, \cite{HS} for Nichols algebras.

\subsection{Yetter-Drinfeld modules and Nichols algebras}\label{subsec:YD-Nichols}
Let $H$ be a Hopf algebra. A (left) Yetter-Drinfeld module over $H$ is a left 
$H$-module $(V,\cdot)$ and a left $H$-comodule $(V,\delta)$ such that
$$ 
\delta(h\cdot v) = h_{(1)}v_{(-1)}S(h_{(3)})\ot h_{(2)}\cdot v_{(0)}\qquad \text{ for all }h\in H, v\in V.
$$
Yetter-Drinfeld modules together with morphisms of left $H$-modules and left $H$-comodules form a braided rigid tensor category denoted by $\ydh$.
The braiding is given by $c_{V,W}(v\ot w) = v_{(-1)}\cdot w \ot v_{(0)}$ for all $v\in V$, $w\in W$ with $V$, $W$ objects in $\ydh$.
{The Hopf algebra $H$ is
an object in $\ydh$ by the left adjoint action on itself and the coaction given by the comultiplication.}

Let $V \in \ydh$. Then, the tensor algebra $T(V)$ is a graded braided
Hopf algebra in $\ydh$. 
The \textit{Nichols algebra} 
$\B(V) = \bigoplus_{n\geq 0} \B^{n}(V)$ of $V$ is the graded braided Hopf algebra in $\ydh$ defined by the quotient
$\B(V) = T(V) /\J(V)$,
where $\J(V)$ is the largest Hopf ideal of $T(V)$ generated as an ideal by 
homogeneous elements of degree bigger or equal than $2$. By definition, we have that 
$\B^{0}(V) = \Bbbk$ and $\B^{1}(V) = V$.
Actually, one can define a Nichols algebra 
$\B(V)$ from any rigid braided vector space $(V,c)$; it turns out that $\B(V)$ is completely determined, as algebra and coalgebra,
by the braiding.
There are several equivalent definitions of the Nichols algebra associated with a braided vector space $(V,c)$, each of them 
particularly useful for different purposes. Here below we recall the one related to the quantum symmetrizer, as it enables the computation of \textit{at least} some relations.

Let $V$ be a vector space and $c \in \End(V\ot V) $ be a solution of the \textit{braid equation}, that is 

$$
(c\ot \id)(\id\ot c)(c\ot \id)=(\id\ot c)(c\ot \id)(\id\ot c) \qquad\text{ in }\End(V\ot V\ot V).
$$ 

Let $T(V)$, $T^{c}(V)$ be the tensor algebra and 
the cotensor algebra of $V$, respectively. Both are braided bialgebras and there exists a unique 
bialgebra map $\mathbf{S}:T(V)\to T^{c}(V)$ such that $\mathbf{S}|_V = \id_{V}$. The image
$\Img \mathbf{S} \subseteq T^{c}(V)$ is a braided bialgebra called the \textit{quantum symmetric algebra}.
If the braiding is rigid, then $\Img \mathbf{S} = \B(V)$ is a Nichols algebra. 
There exists a way to describe explicitly the kernel of $\mathbf{S}$ by means of actions of braid groups.  

The braid group  
$$\mathbb{B}_n=\langle \tau_{1},\ldots, \tau_{n-1}|\ \tau_{i}\tau_{j}=\tau_{j}\tau_{i}, \
\tau_{i+1}\tau_{i}\tau_{i+1}=\tau_{i}\tau_{i+1}\tau_{i}, \text{ for }1\leq i\leq n-2\text{ and }j\neq i\pm1\rangle$$ 
acts
naturally on $V^{\ot n}$ via $\rho_{n}:\mathbb{B}_{n} \to GL(V^{\ot n})$ with $\rho_{n}(\tau_{i})=c_{i} 
=\id_{V^{\ot i-1}}\ot c\ot\id_{V^{n-i-1}}:V^{\ot n}\to V^{\ot n} $. Using the
Matsumoto
(set-theoretical) section from the symmetric group $\mathbb{S}_{n}$ to $\mathbb{B}_n$:
\[
M:\mathbb{S}_n\to \mathbb{B}_n,\qquad
(i,i+1)\mapsto \tau_i,\qquad \text{for all }1\leq i\leq n-1,\]
one can define the quantum symmetrizer
$QS_n:V^{\ot n}\to V^{\ot n}
$ by
$$
QS_{n} = \sum_{\sigma \in \mathbb{S}_{n}}\rho_{n}(M (\sigma)) \in \End(V^{\ot n} ).
$$
For example $QS_{2} = \id + c$, and 
\[
QS_{3}=\id+c\ot \id+\id\ot c+(\id\ot c)(c\ot \id)+(c\ot \id)(\id\ot c)+(c\ot \id)(\id\ot c)(c\ot \id).
\]

The Nichols algebra associated with $(V,c)$
is the quotient of the tensor algebra $T(V)$  by the homogeneous ideal
$$
\mathcal{J}=\bigoplus_{n\geq 2 } \mathcal{J}_{n} = \bigoplus_{n\geq 2 } \Ker QS_{n},
$$
 or equivalently,
$\B(V):= \B(V,c)=\oplus_n\Img(Q{S}_n) = \oplus_n T(V) / \mathcal{J}_{n} $. In particular,
$\B(V)$ is a graded algebra.

If $W\subseteq V$ is a subspace such that $c(W\ot W) \subseteq W\ot W$, one may identify
$\B(W)$ with a subalgebra of $\B(V)$; eventually  belonging to different braided rigid categories.
In particular, if $\dim \B(W) = \infty$, then $\dim \B(V) = \infty$. Thus, if 
$V$ contains a non-zero element $v$ such that $c(v\ot v) = v\ot v$, then $\dim (V) = \infty$.
We refer to \cite{A}, \cite{HS} for more details on Nichols algebras.

\section{The Hopf algebra $K_{n}$}\label{sec:Kn}

We fix groups $N$ and $Q$ and a \emph{right} action of $Q$ on $N$ by $N\times Q\to N$, $(u,x)\mapsto u^x$.
This translates into a \emph{left} action of $\kk Q$ on $\kk^N=(\kk N)^*$ by $(^x f)(u)=f(u^x)$ for all $x\in Q$ and $f\in \kk^{N}$.
For $u\in N$, write $p_u\in \kk^N$ 
for the map given by $p_u(v)=\delta_{u,v}$. Then 
$\{p_{u}\}_{u\in N}$ is
the dual basis of the standard basis of $\kk N$. The action of $Q$ on this basis is then given by $ ^x p_u = p_{u^{x^{-1}}}$ for all $x\in Q$, $u\in N$.

\subsection{The Hopf algebra $\kk^N\rtimes_\beta \kk Q$}
Let $\beta\colon Q\times N\times N\to \kk^\times$ be a map. For $x\in Q$ we write $\beta_x(u,v)=\beta(x,u,v)$ 
so that we may consider $\beta_x\colon N\times N\to \kk^\times$ as an element in $\kk^{N\times N}$. Clearly, one also has an action of $\kk Q$ on 
$\kk^{N\times N}$; for short,   
we also abbreviate 
$ (^x \beta_y)(u,v)=\beta_y(u^x,v^x)$ 
for all $x,y\in Q$ and $u,v\in N$.  

\begin{definition}
 We say that $\beta$ is a normalized $2$-cocycle if
for $x,y\in Q$ and $u,v,w\in N$ we have 
\begin{eqnarray*}
\beta(1_Q,u, v)&=&1,\\
\beta(xy, u, v)&=&\beta(x,u,v)\beta(y,u^x,v^x),\\
\beta(x,1_N,v)&=& 1 = \beta(x,u,1_N),\\
\beta(x, v,w)\beta(x,u,vw)&=&
\beta(x,uv,w)\beta(x,u,v).
\end{eqnarray*} 
\end{definition}

In short $\beta$ can be viewed as a normalized $1$-cocycle as a map from $Q$ to 
$\Map(N\times N,\kk^\times)$ with respect to the induced action discussed above (i.e., $\beta_1=\varepsilon$ and $\beta_{xy}=\beta_x\, {}^x \beta_y$) and for each fixed $x$, 
$\beta_x$ is a normalized $\kk^\times$-valued group $2$-cocycle on $N$ with respect to the trivial action.  
We will be mostly focused on the special case where for each $x\in Q$, the map $\beta_x$ is a bicharacter, i.e., for $x\in Q$ and $u,v,w\in N$ we have that
$\beta_x(uv,w)=\beta_x(u,w)\beta_x(v,w)$ and $\beta_x(u,vw)=\beta_x(u,v)\beta_x(u,w)$.

Using the normalized $2$-cocycle $\beta$ we may define a Hopf algebra
structure on $\kk^{N}\ot \kk Q$ as follows.

\begin{definition}
The Hopf algebra $B=\kk^N\rtimes_\beta \kk Q$ is the vector space with basis $\{p_u\widehat{x}:u\in N, x\in Q\}$, whose
multiplication is given by 
$$
(p_u\widehat{x})(p_v\widehat{y}) = \delta_{u,v^{x^{-1}}} p_u\widehat{xy}
\qquad\text{ for all }x,y\in Q, u,v \in N.
$$
The comultiplication is given by
$$
\Delta(p_u\widehat{x})=\sum_{v,w\in N, vw=u} \beta_x(v,w)p_v\widehat{x}\otimes p_w \widehat{x};
$$
in particular, $\Delta(\widehat{x}) =\sum_{v,w\in N} \beta_x(v,w)p_v\widehat{x}\otimes p_w \widehat{x}$.
{The counit is given by 
$$
\varepsilon(p_{u}) = p_{u}(1) = \delta_{u,1} \qquad\text{ and }\qquad \varepsilon(\widehat{x}) = 1\qquad\text{ for all }u\in N, x\in Q.
$$
} 
The antipode is given by:
\begin{eqnarray*}
S(p_u) &=& p_{u^{-1}},\\
S(\widehat{x}) &=& \sum_{u\in N} \beta_x^{-1}(u,u^{-1})\widehat{x^{-1}} p_u
= \sum_{u\in N} \beta_x^{-1}(u,u^{-1})p_{u^x}\widehat{x^{-1}}.
\end{eqnarray*}
In the special case when every $\beta_x$ is an alternating bicharacter we have $S(\widehat{x})=\widehat{x^{-1}}$
for all $x\in Q$.
\end{definition}

We frequently make the following identifications for $f\in \kk^N$ and $x\in Q$:
\begin{eqnarray*}
1&=&1_B=\widehat{1_Q},\\
f&=&f\widehat{1_Q}=\sum_{u\in N} f(u)p_u\widehat{1_Q}, \\
f\widehat{x} &=& \sum_{u\in N} f(u)p_u\widehat{x}, \\
\widehat{x}&=&\varepsilon_N \widehat{x} = \sum_{u\in N} p_u\widehat{x}.
\end{eqnarray*}
With these identifications in mind, $\kk^N$ is a subalgebra of $B$ and $\widehat{x} f = (^x f) \widehat{x}$ for $f\in\kk^N, x\in Q$.

\subsection{The Hopf algebra $K_{n}$}
Assume that $n$ is odd and bigger than 1 and that $m$ divides $n$.  Let $\xi$ be a fixed primitive $m$-th root of $1$.
The Hopf algebras $H_{m,n}$ were first described by G.~I.~Kac {\color{magenta} \cite{K}}
and later on revisited by A.~Masuoka {\color{magenta} \cite{Ma}}.  
The presentation below is taken from {\color{magenta} \cite{M}}.  They are a special case of the construction above where
$$
N=C_n\times C_n = \langle a,b : a^n, b^n\rangle,\qquad\text{ and }\qquad Q=C_2=\langle x : x^2 \rangle,
$$ 
the action of $Q$ on $N$ is given by $a^x=b, b^x=a$, and the cocycle
$\beta$ is an alternating bicharacter given by
\smallskip
\begin{equation}\label{eq:definition-beta}
\beta_x(a^i b^j, a^k b^\ell) = \xi^{i\ell-jk}. 
\end{equation}
\smallskip
If $m=n$, then we set $K_n=H_{n,n}$.
For a map $f\in\kk^N$ we write $f(i,j)=f(a^ib^j)$ and for a map $g\colon \kk^{N\times N}\simeq \kk^N\otimes\kk^N\to\kk$ 
we sometimes abbreviate $g((i,j),(k,\ell))=g(a^ib^j\otimes a^kb^\ell)$.

\subsection{Structure of $K_n$}
Let $n>1$ be odd and let $\xi$ be a fixed primitive $n$-th root of $1$.  Set $p_{i,j}=p_{a^i b^j}$ and $f_{i,j}=p_{i,j}\widehat{x}$ for all $i,j\in \Zn$.  Then 
$\{p_{i,j}, f_{i,j}: i,j\in \mathbb{Z}_n\}$ is a basis for $K_n$.
The algebra structure in terms of this basis is as follows:
$$
p_{ij}p_{ij} = p_{ij},\qquad
p_{ij}f_{ij} = f_{ij}, \qquad
f_{ij}p_{ji} = f_{ij},\qquad
f_{ij}f_{ji} = p_{ij},
$$
where all other products of two basis elements are zero.  The coalgebra
structure is given by:
\begin{eqnarray*}
\Delta(p_{ij})&=& \sum_{i'+i''=i, j'+j''=j} p_{i'j'}\otimes p_{i''j''},\qquad \varepsilon(p_{ij}) = \delta_{i,0}\delta_{j,0},\\
\Delta(f_{ij})&=&  \sum_{i'+i''=i, j'+j''=j} \xi^{i'j''-j'i''}\, f_{i'j'}\otimes f_{i''j''},\qquad \varepsilon(f_{ij})= \delta_{i,0}\delta_{j,0},\\
\Delta(\widehat{x}) &=&\sum_{i,j,k,\ell\in \Zn} \xi^{i\ell-jk}p_{ij}\, \widehat{x}\otimes p_{k\ell} \widehat{x},\qquad \varepsilon(\widehat{x}) = 1.
\end{eqnarray*}
The antipode is as follows:
\begin{align*}
S(p_{ij})=p_{-i,-j},\qquad \qquad
S(f_{ij})=f_{-j,-i},\qquad \qquad
S(\widehat{x}) =\widehat{x^{-1}}. 
\end{align*}

\section{Simple Yetter-Drinfeld modules}\label{sec:simple-YD-modules}

{ In this section we present all simple Yetter-Drinfeld modules over the Hopf algebra $K_{n}$.
First we adapt the method of ``little groups" of Wigner and Mackey to produce simple
Yetter-Drinfeld modules over $B=\kk^N\rtimes_\beta \kk Q$  from one-dimensional comodules. Then,
we construct simple objects from a matrix coalgebra coaction.}

\subsection{Little groups of Wigner and Mackey}
In the following we describe an adaptation of the method of ``little groups" of Wigner and Mackey.  
This method is used to describe irreducible representations of 
a semidirect product of groups $A\rtimes H$ with $A$ abelian. 
The treatment below is taken from Subsection 8.2 of \cite{Serre}.  
Note that in its proof it is not needed for $A$ to be a group; the treatment and proofs carry over 
almost word for word to describe irreducible representations of an algebra $B=A\rtimes \kk Q$ where $Q$ 
is a finite group acting on the finite dimensional commutative semisimple algebra $A$.  
Using this action one has that $Q$ also acts on the left on $X=\Alg(A,\kk)$ by 
$(q\chi)(a)=\chi(q^{-1}\cdot a)$ for all $q\in Q$, $a\in A$ and $\chi\in X$.

Let $\chi_1,\ldots, \chi_k$ be representatives of all distinct orbits of $X/Q$.  
For $i=1,\ldots, k$, let $Q_i$ be the stabilizer of $\chi_i$, i.e., $Q_i=\{q\in Q: q\chi_i=\chi_i\}$, 
and let $B_i=A\rtimes\kk Q_i$.  
For $i=1,\ldots, k$ and $\rho\colon Q_i\to \GL(U)$, let $\chi_i\otimes \rho\colon B_i\to \GL(U)$ denote the 
representation of $B_i$ given by $(\chi_i\otimes \rho)(a\rtimes q)=\chi_i(a)\rho(q)$ for $a\in A$ and $q\in Q_i$.  Finally,
let $\theta_{i,\rho}\colon B\to \GL(B\otimes_{B_i} U)$ be the induced representation.

\begin{theorem}[cf. Proposition 25 of \cite{Serre}]\
\begin{enumerate}
\item The representation $\theta_{i,\rho}$ is irreducible if and only if $\rho$ is irreducible.
\item The representations $\theta_{i,\rho}, \theta_{i',\rho'}$ are equivalent if and only if $i=i'$ and 
the representations $\rho, \rho'$ of $Q_i$ are equivalent.
\item Every irreducible representation of $B$ is equivalent to some $\theta_{i,\rho}$.
\end{enumerate}
\end{theorem}
\qed

\begin{remark}{\em If we do not fix representatives of orbits, then we can describe 
$\theta_{\chi,\rho}$ where $\chi\in\Alg(A,\kk)$ and $\rho$ is an irreducible 
representation of $\Stab_Q(\chi)$ in the obvious way.  Then two representations 
$\theta_{\chi,\rho}$ and $\theta_{\chi',\rho'}$ are equivalent if and only if the following happens:
\begin{enumerate}
\item The orbits of $\chi$ and $\chi'$ under the action of $Q$ are equal.
\item If $q\in Q$ is such that $\chi'=q\chi$, then $q \Stab_Q(\chi) q^{-1}= \Stab_Q(\chi')$.  
Via this identification we can consider $\rho'$ as a representation of $\Stab_Q(\chi)$
and in this sense it should be equivalent to $\rho$.
\end{enumerate} }
\end{remark}

\begin{remark}{\em We can describe $\theta_{i,\rho}$ in a more explicit way as follows: 
Let $Q_i\le Q$ be the stabilizer of $\chi_i$, and let $U$ be the simple $\kk Q_i$-module 
corresponding to an irreducible representation $\rho$ of $Q_i$. 
Pick representatives $q_j=q_{j,i}$, $j=1,\ldots, m$, of cosets $Q/Q_i$.  
Then the representation $\theta_{i,\rho}$ corresponds to the simple $B$-module 
$W=W_{i,\rho}=\bigoplus_{j=1}^m U_j$ where $U_j$ is $U$ as an $\kk Q_i$-module.  
The action of $A$ on $U_j$ is given by $a\cdot u = \chi_i(q_j^{-1}\cdot a) u = \chi_i(q_j^{-1}a q_j) u$.  
The action of $q\in Q$ is as follows: there is unique $j\in\{1,\ldots, m\}$ and $q'\in Q_i$ such that $q=q_j q'$.  
Then for $u\in U_\ell$ we have that $q\cdot  u = \rho(q_{\ell}^{-1}q'q_{\ell}) u\in U_{j\triangleright \ell}$, 
where $j\triangleright \ell$ is the unique index such that $q_j q_\ell \in q_{j\triangleright \ell} Q_i$.}
\end{remark}

\subsection{Simple Yetter-Drinfeld modules induced by one-dimensional comodules}
Let $B=\kk^N\rtimes_\beta \kk Q$ and assume furthermore that $\beta$ is a bicharacter.  
Below we apply the theory of little groups discussed above to describe simple objects in 
$ _{B} ^{\widehat{N}}  \mathcal{Y}\mathcal{D} $, 
the subcategory of $\YD{B}{B} $ consisting of those Yetter-Drinfeld modules $V$ 
whose coaction lies inside $ \kk\widehat{N} \otimes V$.   
{ The idea is to 
define a simple $B$-module with 
a compatible homogeneous coaction 
on 
$\kk \widehat{N}$. 
}

Consider the action of $Q$ on 
$N\times \widehat{N}$ given by
$$
x*(a, \chi) = (a^{x^{-1}}, \beta_x(-,a^{x^{-1}})\beta^{-1}_x(a^{x^{-1}},-) (^x \chi))\qquad \text{ for all }
x\in Q, a\in N, \chi \in \widehat{N},
$$
and let $(a_1,\chi_1),\ldots ,(a_k,\chi_k)$ be a fixed set of representatives of distinct orbits under this action.  
For each $i=1,\ldots, k$, let $Q_i=\Stab_Q(a_i,\chi_i)$ and let $U$ be an irreducible representation of $Q_i$.  
Then the induced $\kk Q$-module $\Theta(U,a_i,\chi_i)=\kk Q\otimes_{\kk Q_i} U$ 
becomes an element in $ _{B} ^{\widehat{N}}  \mathcal{Y}\mathcal{D} $ 
as follows: for all $x,y\in Q$, $f\in \kk^{N}$ and $u\in U$ we set
\begin{eqnarray*}
(f\widehat{x})\cdot (y\otimes_{\kk Q_i} u) &=&  f(a_i^{xy})(xy \otimes_{\kk Q_i} u),\\ 
\delta(y\otimes_{\kk Q_i} u) &=&
\beta_y(-,a_i^{y^{-1}})\beta_y^{-1}(a_i^{y^{-1}},-) (^y \chi_i)  \otimes (y\otimes_{\kk Q_i} u).
\end{eqnarray*}
{ In particular, 
$(f\widehat{x})\cdot (1\otimes_{\kk Q_i} u) =  f(a_i^{x})(x \otimes_{\kk Q_i} u)$ and 
$\delta(1\otimes_{\kk Q_i} u) =
\chi_i \ot 1\otimes_{\kk Q_i} u$ for all $u\in U$. Note that the formula for the coaction follows from the compatibility condition, i.e. 
$\delta(y\otimes_{\kk Q_i} u) = \delta(\widehat{y} \cdot (1 \otimes_{\kk Q_i} u))$.}

Alternatively, pick representatives $x_1,\ldots, x_m$ of cosets $Q/Q_i$.  Then
$\Theta(U,a_i,\chi_i)=\bigoplus_{j=1}^m U_j$ as a $\kk Q_i$-module, 
where as a $\kk Q_i$-module we have that each $U_j\simeq U$.  
Let us describe its structure explicitly.
For each $j=1,\ldots,m$, let $v_j=u_i^{x_j^{-1}}$ and let
$\theta_j=\beta_x(-,v_j)\beta^{-1}_x(v_j,-)( ^{x_j} \chi_i) \in \widehat{N}$.
\smallskip
\begin{itemize}
 \item The $\kk \widehat{N}$-coaction on $U_j$ is then given by $\delta(u)=\theta_j\otimes u$.
 \item   The $B$-action on $U_j$ is given as follows:  If $f\in \kk^N$, then $f\cdot u=f(v_j) u$.   
 If $x\in Q$, let $k\in\{1,\ldots, m\}$ and $y\in Q_i$ be unique such that $x x_j = x_k y$; then $\widehat{x}\cdot u$ is the element corresponding to $y\cdot u$ in $U_k$.
\end{itemize}

{The Yetter-Drinfeld modules  $\Theta(U,a_i,\chi_i)$ are then simple because they are simple as $B$-modules 
by the \textit{little groups} construction. Note that the dimension of the module depends on the dimension of $U$ and the size of the orbit by the action of $Q$.´}

\begin{remark}{\em
If $\beta$ is ``partially trivial", then the above can be extended as follows.
Let 
$$R:=\{t\in Q: \forall a,b\in N, \forall x\in Q, \beta_t(a,b)=1, \beta_x(a^t,b)=\beta_x(a,b)=\beta_x(a,b^t)\};
$$
i.e., $R$ 
consists of elements $t$ of $Q$ where $\beta_t$ is trivial and the actions on each components of $\beta_x$ are trivial as well.  
It turns out that $R$ is a normal subgroup of $Q$ and that $\kk \widehat{N}\rtimes R\subseteq G(B)$.  
In general the inclusion is strict; in the special cases where $\beta$ is either trivial (implying that $R=Q$), 
or $\beta$ is non-degenerate (in the sense that for every $a\not=1_N$ and every $x\not=1_Q$ 
we have that characters $\beta_x(-,a)$, $\beta_x(a,-)$ have trivial kernels; consequently $R=1$) we get equalities. 
Define an action of $Q$ on $N\times (\widehat{N}\rtimes R)$ by
$$
x*(a, \chi r) = (a^{x^{-1}}, \beta_x(-,a^{x^{-1}})\beta_x(a^{x^{-1}},-) (^x \chi)(xrx^{-1})).
$$

Let $(a_1,\chi_1 r_1),\ldots ,(a_k,\chi_k r_k)$ be a fixed set of representatives of distinct orbits under this action.  
For each $i=1,\ldots, k$, let $Q_i=\Stab_Q(a_i,\chi_i r_i)$ and let $U$ be an irreducible representation of $Q_i$.  Then the induced $\kk Q$ module 
$\Theta(U,a_i,\xi_i)=\kk Q\otimes_{\kk Q_i} U$ becomes a Yetter-Drinfeld modules over B as follows:
\begin{eqnarray*}
(f\widehat{x})\cdot (y\otimes_{\kk Q_i} u) &=&  f(a_i^{xy})(xy \otimes_{\kk Q_i} u),\\ 
\delta(x\otimes_{\kk Q_i} u) &=& \beta_x(-,a_i^{x^{-1}})\beta_x^{-1}(a_i^{x^{-1}},-) (^x \chi_i) \widehat{xrx^{-1}} \otimes (x\otimes_{\kk Q_i} u).
\end{eqnarray*}}
\end{remark}

\subsection{Simple Yetter-Drinfeld modules over $K_n$ induced by 
one-dimensional subcomodules}\label{sec:structureVandU}

Here we apply the recipe discussed above to the case where $B=K_n$ to describe all simple Yetter-Drinfeld modules in 
$\YD{K_{n}}{\widehat{N} } $. Recall that $N = \langle a,b:\ a^{n},b^{n}\rangle \simeq
C_{n}\times C_{n}$, $Q=C_{2}$ and $\beta_{x}(a^{i}b^{j},a^{k}b^{\ell}) = \xi^{i\ell - jk}$ for all $i,j,k,\ell \in \Zn$.

  For $m,\ t\in\mathbb{Z}_n$, let $\chi_{m,t}\in\widehat{N}$ be the character on $N$ given by $\chi_{m,t}(a^i b^j)=\xi^{mi+tj}$.  
  Note that the action of $x$ on $\chi_{m,t}$ is given by $ ^x \chi_{m,t} = \chi_{t, m}$ and the ``twisted" action of $C_{2}$ on $N\times \widehat{N}$ 
  is given by 
  $$
  x*(a^{i}b^{j}, \chi_{m,t})=(a^j b^i, \chi_{t+2i,m-2j}) \qquad \text{ for all }i,j,m,t\in \Zn.
  $$  

  Then, the orbits under the action of $Q$ are as follows:
 
 \smallskip
 \begin{enumerate}
\item Orbits of size one: $\{(a^ib^i, \chi_{m,m-2i})\}$ for $i,m\in\Zn$. 
\bigskip
\item Orbits of size two:
\begin{enumerate}
\smallskip
\item[$(a)$] $\{(a^i b^i, \chi_{m,t}), (a^ib^i, \chi_{t+2i, m-2i}\}$, where $i,m,t\in\Zn$ and $t\not=m-2i$.
\smallskip
\item[$(b)$] $\{(a^ib^j, \chi_{m,t}), (a^jb^i,\chi_{t+2i,m-2j})\}$, where
$i,j,m,t\in\Zn$ and $i\not=j$.  
\end{enumerate}
\end{enumerate}
\smallskip
We remark that in the case $(b)$, it is impossible to have 
$(m,t)={(t+2i,m-2j)}$.

\bigskip
The corresponding simple Yetter-Drinfeld modules are as follows:
\smallskip
\begin{enumerate}
\item[$\mathbf{(V^\epsilon_{i,m})}$.]  For $\epsilon =\pm 1$ and $i,m\in \Zn$, the objects $V^\epsilon_{i,m} \in\  \YD{K_{n}}{\widehat{N}} $  
are one-dimensional vector spaces generated by $v\neq 0$ where
\smallskip
\begin{itemize}
 \item[$\triangleright$] the coaction is given by $\delta(v)=\chi_{m,m-2i}\otimes v$;
 \smallskip
 \item[$\triangleright$] the action of is given by $(f\widehat{x}^{k})\cdot w=f(a^i b^i) \epsilon^{k}w$ for $f\in \kk^N$ and $k=0,1$;  
 \smallskip
 \item[$\triangleright$] the braiding is given by $c(v\otimes v)=\xi^{2i(m-i)}\, v\otimes v$. 
\end{itemize}
\smallskip
Up to isomorphism, there are $2n^2$ such modules.



\bigskip

\item[$\mathbf{(U_{i,j,m,t})}$.] For $i, j, m,t \in \Zn$, the objects $U_{i,j,m,t}$ are two-dimensional vector spaces spanned by non-zero vectors $u_1, u_2$ where
\smallskip
\begin{itemize}
 \item[$\triangleright$] The coaction is given by $\delta(u_1)=\chi_{m,t}\otimes u_1$, 
$\delta(u_2)={\chi_{t+2i,m-2j}}\otimes u_2$.  
\smallskip
\item[$\triangleright$] The action is determined by
$f\cdot u_1 = f(a^ib^j)\cdot u_1$,  $f\cdot u_2 = f(a^jb^i)\cdot u_2$ for $f\in \kk^N$ and 
$\widehat{x}\cdot u_1=u_2$, $\widehat{x}\cdot u_2=u_1$.  
\smallskip
\item[$\triangleright$] The braiding is given by
\begin{align*}
c(u_1\otimes u_1)&= \xi^{mi+tj} u_1\otimes u_1, & 
c(u_1\otimes u_2)&= \xi^{it+mj} u_2\otimes u_1, \\ 
c(u_2\otimes u_1)&= \xi^{it+mj + 2(i^{2}-j^{2})} u_1\otimes u_2, &
c(u_2\otimes u_2)&= \xi^{mi+tj} u_2\otimes u_2.
\end{align*}
\end{itemize}
\smallskip

Two such modules $U_{i,j,m,t}$ and $U_{i',j',m',t'}$ are isomorphic if and only if 
$(i',j',m',t')\in\{(i,j,m,t), (j,i,t+2i,m-2j)\}$. Note that if $i\not=j$, then it is impossible to have both $m=t+2i$ and $t=m-2j$.

\smallskip
These modules are reducible if and only if $i=j$ and $t=-2i+m$.  If this happens then $U_{i,i,m,-2i+m}\simeq V^+_{i,m}\oplus V^-_{i,m}$ 
where the isomorphism is given by $u_1\mapsto v^+ + v^-$ and $u_2\mapsto v^+-v^-$, being $v^{\pm}$ the generator 
of $V_{i,m}^{\pm}$, respectively.

\smallskip
Up to isomorphism that are $\frac{1}{2} n^3(n-1) + \frac{1}{2} n^2(n-1)$ such 
simple modules.
\end{enumerate}

\bigskip
The sum of the squares of dimensions of these simple Yetter-Drinfeld modules is equal to
\begin{equation}\label{eq:dimension-simple-kNB}
n^2\cdot 1+ n^2\cdot 1 + \frac{1}{2} n^2(n-1)\cdot 4 + \frac{1}{2} n^3(n-1)\cdot 4 = 2n^4=\dim(B)\cdot\dim(\kk \widehat{N}).
 \end{equation}

\bigskip
\subsection{Simple Yetter-Drinfeld modules over $K_n$ with matrix coalgebra coaction}\label{sec:structureW}

For $i,j\in \Zn$, we define the following elements 
in $K_{n}$

$$
e_{ij}=\sum_{k\in\Zn} \xi^{-2(i+j)k} f_{k+i-j,k-i+j}.
$$

\begin{proposition} 
The collection $\{e_{ij}\}_{i,j\in\Zn}$ is linearly independent.
\end{proposition}

\begin{proof}
Suppose $\sum_{i,j} \lambda_{ij} e_{ij}=0$.  Then, for a fixed
$r,s\in\Zn$, the coefficient of $f_{rs}$ in this sum is 
$$
\sum_{2(i-j)=r-s} \lambda_{ij} \xi^{-(i+j)(r+s)}.
$$
Write $2^{-1}$ for the multiplicative inverse of $2$ in $\Zn$ (i.e., $2^{-1} = \frac{n+1}{2}$).  
Now fix $k,\ell\in \Zn$ and set $r=2^{-1}(2k-2^{-1}\ell)$ and $s=-2^{-1}(2k+2^{-1}\ell)$ so that $r-s=2k$ and $r+s=-2^{-1}\ell$.  
Then this coefficient becomes
$$
\sum_{i\in\Zn} \lambda_{i,i-k} \xi^{i\ell}.
$$
Since the elements $\{f_{ij}\}_{i,j\in \Zn}$ are linearly independent,
we have that $\xi^\ell$ is a root of the polynomial $p(x)=\sum_{i=0}^{n-1} \lambda_{i,i-k} x^i$ for every $\ell \in \Zn$.  
This means that $p$ must be identically zero and hence we have that $\lambda_{i,i-k}=0$ for all $i,k$.
\end{proof}

The following proposition gives the comultiplication of the elements 
$\{e_{ij}\}_{i,j\in \Zn}$; they constitute a \textit{comatrix basis}.

\begin{proposition}\label{prop:comatrix-basis} For all 
$i,j\in \Zn$
we have 
$$
\Delta (e_{ij}) = \sum_r e_{ir}\otimes e_{rj}\qquad\text{ and }\qquad
{\varepsilon(e_{ij}) = \delta_{i,j}.}
$$
\end{proposition}

\begin{proof}
Fix $i,j\in \Zn$. A direct computation yields
\begin{eqnarray*}
\Delta (e_{ij}) &=& \Delta \big(\sum_k \xi^{-2(i+j)k} f_{k+i-j,k-i+j}\big) \\
&=& \sum_{k,\ell,m} \xi^{-2(i+j)k+\ell(k-i+j-m)-m(k+i-j-\ell)} f_{\ell,m}\otimes f_{k+i-j-\ell, k-i+j-m}.
\end{eqnarray*}
We now introduce new variables $s,t,r\in\Zn$ and use the following changes 
$$
\ell = t+i-r, \qquad\qquad
m = t-i+r, \qquad\qquad
k = s+t.
$$
Since
$$
2t = \ell+m,\qquad\qquad
2r = m-\ell-2i,\qquad\qquad
2s = 2k-\ell-m,
$$
and $2$ is invertible in $\Zn$, this change of variable is reversible.  Under this change,
the sum above is equal to
\begin{eqnarray*}
\sum_{r,s,t} \xi^{-2(i+j)t-2(r+j)s} f_{t+i-r,t-i+r}\otimes f_{s+r-j, s-r+j} = \sum_e e_{ir}\otimes e_{rj}.
\end{eqnarray*}
{ Finally, $\varepsilon(e_{ij}) = \sum_k \xi^{-2(i+j)k} \varepsilon(f_{k+i-j,k-i+j}) = \sum_k \xi^{-2(i+j)k} \delta_{k,j-i}\delta_{k,i-j} = \delta_{i,j}$.}
\end{proof}

\begin{corollary}
The coalgebra $\kk^N \widehat{x}$ is isomorphic to $\M_n(\kk)^*$, the simple matrix coalgebra of dimension $n^{2}$.
\qed
\end{corollary}

The following technical lemmas will help us to describe the 
$K_{n}$-module structure on the linear span of the elements $\{e_{r0}\}_{r\in \Zn}$. 
{ As it is a subcoalgebra of $K_{n}$, this is given by the adjoint action of 
$K_n$ on itself, i.e. $y\hit z = y_{(1)} z S(y_{(2)})$ for all $y,z\in K_{n}$. 
For example, a quick check yields
that for the elements $f_{ij}$ with $i,j\in \Zn$ 
and $\chi\in\kk^N$ a character (i.e., a grouplike in $\kk^N\subseteq K_{n}$),  we have  

\begin{eqnarray}
\label{eq:hitx}\widehat{x}\hit f_{ij} &=& \xi^{i^2-j^2}f_{ji},\\
\label{eq:hitchi}\chi\hit f_{ij} &=& \chi(a^{i-j}b^{j-i}) f_{ij},
\end{eqnarray}
}

\begin{lemma}
For $p,q,r\in\Zn$ we have that
$$
f_{pq}\hit e_{r,0}=\begin{cases}
e_{-r,0}, & p=-2r, q=2r \\
0, &\mbox{otherwise}
\end{cases}
$$
\end{lemma}

\begin{proof}
A direct computations yields that
$$
f_{pq}\hit f_{rs} = \begin{cases} \xi^{r^2-s^2} f_{sr}, &p=-r+s, q=r-s \\ 0, &\mbox{otherwise}.
\end{cases}
$$
Hence
\begin{align*}
f_{-2r,2r}\hit e_{r0} &= f_{-2r,2r} \hit\sum_k \xi^{-2jk} f_{k+r,k-r}
= \sum_k \xi^{-2rk+(k+r)^2-(k-r)^2} f_{k-j,k+j} \\
&= \sum_{k} \xi^{2rk} f_{k-r,k+r} = e_{-r,0}.
 \end{align*}
A similar computation also shows that for $(p,q)\not=(-2r,2r)$ we get $f_{pq}\hit e_{r,0}=0$.
\end{proof}

\begin{lemma}
For $m,t,i,j\in\Zn$ we have that
$$
\chi_{m,t}\hit e_{i,j} = \xi^{2(m-t)(i-j)} e_{i,j}.
$$ 
In particular
$$
\chi_{1,-1}\hit e_{r,0} = \xi^{2r} e_{r,0}.
$$
\end{lemma}

\begin{proof}
Recall from \eqref{eq:hitchi}
that for any character $\chi\in\kk^N$ we have that
$$
\chi\hit f_{rs} = \chi(r-s,-r+s) f_{rs}.
$$
Hence
\begin{align*}
\chi_{m,t}\hit e_{ij} &= \chi_{m,t}\hit \sum_k \xi^{-2(i+j)k} f_{k+i-j,k-i+j} 
= \sum_k \chi_{m,t}(2(i-j),-2(i-j)) \xi^{-2(i+j)k} f_{k+i-j,k-i+j} \\
&= \xi^{2(m-t)(i-j)} e_{ij}.
\end{align*}
\end{proof}

{ Using the results above, and the fact that the elements $\{f_{ij}\}_{i,j\in \Zn}$ and the characters 
$\{\chi_{m,t}\}_{m,t\in \Zn}$ span linearly $K_{n}$,
we obtain the description of the 
Yetter-Drinfeld module structure of $W_0=span\{e_{r0}: r\in\Zn\}$.}

\begin{corollary}
The comodule $W_0=span\{e_{r0}: r\in\Zn\}$ is invariant under the adjoint action of $K_n$, i.e., 
it is a Yetter-Drinfeld submodule of the regular Yetter-Drinfeld module $K_n$. 
Its structure is given for all $r \in \Zn$ by 
\begin{enumerate}
 \item[$\triangleright$]  $\delta(e_{r0})= \sum_k  e_{rk}\otimes e_{k0}$;
 \item[$\triangleright$]  $\wx\hit e_{r0} = e_{-r0}$ and $f\hit e_{r0}=f(2r,-2r) e_{r0}$
 for all $f\in \kk^{N}$.
\end{enumerate}
\qed
\end{corollary}

Now, for $i,m\in\Zn$ and $\epsilon \in \{\pm 1\}$ we define the Yetter-Drinfeld modules

\smallskip
$$
W_{i,m}^\epsilon := V^\epsilon_{i,m} \otimes W_0.
$$
\smallskip

Recall from \S \ref{sec:structureVandU} that $V_{i,m}=\kk v$ is a one-dimensional Yetter-Drinfeld module over $K_n$ with 
coaction $\delta(v)=\chi_{m,m-2i}\otimes v$ and
action 
given by 
$f\cdot v = f(i,i) v$
for $f\in\kk^N$ and
$
\widehat{x}\cdot v = \epsilon\, v
$.
Note that this implies that
$$
f_{pq}\cdot v = \begin{cases} \epsilon \, v, & p=q=i \\
0, &\mbox{otherwise}.
\end{cases}
$$
We are considering $W_{i,m}^\epsilon$ as Yetter-Drinfeld submodules of $V\otimes (\kk^N\widehat{x})$ in the obvious way.
For $w\in \kk^N\widehat{x}$, abbreviate $\widetilde{w}=v\otimes w$.  The diagonal action of $B$ on $V\otimes (\kk^N\widehat{x})$ will be denoted by
$$
y\cdot_{i^{\epsilon}}\widetilde{w} = (y_{(1)}\cdot v)\otimes (y_{(2)}\hit w);
$$
whereas the $B$-coaction will be denoted by
$$
\delta_{i,m}(\widetilde{w}) = \chi_{m,m-2i}w_{1}\otimes \widetilde{w_2}.
$$
We first observe that
$$
\delta_{i,m}(\widetilde{e_{r0}}) = \sum_k \chi_{m,m-2i}e_{rk}\otimes \widetilde{e_{k0}}.
$$
Below we compute detailed formulas for the action $\cdot_{i^\epsilon}$.  For $k\in\Zn$, we write
$w_k=\widetilde{e_{k,0}}=v\otimes e_{k,0}$. 
In particular, $W^{\epsilon}_{i,m} = \kk\{w_{0},\ldots,w_{n-1}\}$ as $\Bbbk$-vector spaces and the 
coaction above reads 
\begin{equation}\label{eq:coaction-W}
\delta_{i,m}(w_{r}) = \sum_k \chi_{m,m-2i}\, e_{rk}\otimes w_{k}.
\end{equation}

\begin{lemma}
For $p,q,r\in \Zn$ we have
$$
f_{pq}\cdot_{i^\epsilon} w_r = \begin{cases} \epsilon\, \xi^{4ir} w_{-r}, &p=i-2r, q=i+2r \\ 0, &\mbox{otherwise}
\end{cases}
$$
\end{lemma}

\begin{proof}
A straightforward computation gives
$$
f_{pq}\cdot_{i^\epsilon} \widetilde{e_{r0}} = 
\sum_{t,s}\xi^{t(q-s)-s(p-t)} (f_{ts}\cdot v)\otimes (f_{p-t,q-s}\hit e_{r0}).
$$
For non-zero summands we must have $t=s=i$, $p-t=-2r$, $q-s=2r$ and therefore also $p=i-2r$, $q=i+2r$. 
From this the result immediately follows.
\end{proof}

\begin{lemma}\label{lem:coaction-e}
For $\ell, s,r\in\Zn$ we have
$$
e_{\ell,s}\cdot_{i^{\epsilon}} w_r = 
\begin{cases}
\epsilon\, \xi^{2i(r-\ell)} w_r, & s=\ell+2r \\
0, &\mbox{otherwise}
\end{cases}
$$
\end{lemma}

\begin{proof}
The proof follows by a direct calculation. Indeed,
\begin{eqnarray*}
e_{\ell,s}\cdot_{i^\epsilon} \widetilde{e_{r0}} &=& \sum_k \xi^{-k(\ell+s)} f_{k+\ell-s,k-\ell+s}\cdot_{i^\epsilon} \widetilde{e_{r0}} \\
&=& \epsilon\, \xi^{-2i(\ell+r)+4ir}\, \widetilde{e_{-r,0}}\\
&=& 
\begin{cases} \epsilon\, \xi^{2i(r-\ell)} \widetilde{e_{-r,0}}, & s=\ell+2r, \\
0,\  \mbox{otherwise}.
\end{cases}
\end{eqnarray*}
The second equality follows from the fact that for nonzero terms we must have 
$k+\ell-s=i-2r$, $k-\ell+s=i+2r$ and hence $k=i$, $s=\ell+2r$.
\end{proof}

The next lemma follows by a direct computation. 

\begin{lemma}\label{lem:coaction-chi}
For a character $\chi\in\kk^N$ and $p,q,r\in\Zn$ we have
\begin{eqnarray*}
\chi\cdot_{i^\epsilon} \widetilde{f_{pq}} &=& \chi(i+p-q,i-p+q) \widetilde{f_{pq}},\\
\chi\cdot_{i^\epsilon} w_r &=& \chi(i+2r,i-2r)w_r.
\end{eqnarray*}
\qed
\end{lemma}

\begin{theorem}\label{thm-13}
The $2n$ Yetter-Drinfeld modules $W_{i,m}^\epsilon$, $i,m\in\Zn$ are pairwise non-isomorphic.
{Their Yetter-Drinfeld module structure is given for all $r \in \Zn$ by 

\begin{enumerate}
 \item[$\triangleright$]  $\delta_{i,m}(w_{r}) = \sum_k \chi_{m,m-2i}\, e_{rk}\otimes w_{k}$;
 \item[$\triangleright$]  $\wx\act_{i^\ep} w_r = \ep\,  \xi^{4ir}\, w_{-r}$ and  
$f\act_{i^\ep} w_r = f(i+2r,i-2r) w_r$ 
 for all $f\in \kk^{N}$.
\end{enumerate}
}
\end{theorem}

\begin{proof} 
Note that $W^{+1}_{i,m}$ cannot be isomorphic to $W_{i',m'}^{-1}$ 
as the determinant of the action of $\widehat{x}$ on $W^{+1}_{i,m}$ is $(-1)^{(n-1)/2}$, 
whereas the determinant of the action of $\widehat{x}$ on $W_{i',m'}^{-1}$ is
$-(-1)^{(n-1)/2}$. Indeed, in the ordered basis $w_0,w_j,w_{-j}$, $j=1,\ldots, \frac{n-1}{2}$, $\wx$ is block diagonal: 
the first block is the $1\times 1$ block $[\ep]$, the remaining blocks are $2\times 2$-blocks 
$\ep\begin{pmatrix} 0 & \xi^{-4jr}\\\xi^{4jr} & 0\end{pmatrix}$, $j=1,\ldots, \frac{n-1}{2}$. 

Now assume that $W_{i,m}^\epsilon$ and $W_{i',m'}^\epsilon$ are isomorphic.  
We first note that this implies that $i=i'$ as $\chi_{1,1}$ acts on $W_{i,m}^\epsilon$ and 
$W_{i',m'}$ by multiplication by $\xi^{2i}$ and $\xi^{2i'}$, respectively.
Now suppose that $F\colon W_{i,m}^\epsilon\to W_{i,m'}^\epsilon$ is an 
isomorphism of Yetter-Drinfeld modules.  
Note that the action of $\chi_{1,-1}$ on both of these spaces have eigenvalues $\xi^{2r}$
with corresponding one-dimensional eigenspaces spanned by $w_{r}$.  Hence, $F$ must preserve these eigenspaces, i.e., 
we must have $F(w_{r})=\lambda_r w_{r}$ for non-zero scalars $\lambda_0,\ldots, \lambda_{n-1}$.  
Since $F$ is also a comodule map we must then have that
$\delta_{i,m}(F(w_{0}))=(\operatorname{id}\otimes F)\delta_{i,m'}(w_{0})$.  This gives that
$\lambda_0 \sum_k \chi_{m,m-2i}e_{0k}\otimes w_{k} =
\sum_k \chi_{m',m'-2i}e_{0k}\otimes \lambda_k w_{k}$.  Since $w_{0},\ldots, w_{n-1}$ 
are linearly independent this implies, in particular, that $\chi_{m,m-2i}e_{00} = \chi_{m',m'-2i}e_{00}$.  
The coefficient of $f_{1,1}$ of the left-hand-side of this equation is $\xi^{2m-2i}$ and the 
corresponding coefficient on the right-hand-side is $\xi^{2m'-2i}$.   Therefore $m=m'$.
\end{proof}

\begin{corollary}\label{cor:amountWs}
Every simple Yetter-Drinfeld module over $K_n$ with coaction inside $\kk^N \widehat{x}$ is 
{isomorphic to}
one of $W_{i,m}^\epsilon$ described above.
\end{corollary}

\begin{proof}
The modules $W_{i,m}^\epsilon$ are pairwise non-isomorphic as shown above and clearly simple (they are even simple as comodules). 
Now dimension counting gives
$$
\sum_{i,m,\epsilon} \dim(W_{i,m}^\epsilon) = 2n^2\cdot n^2 = 2n^4 = \dim(K_{n}\otimes \kk^N\widehat{x}).
$$
\end{proof}

\begin{proposition}\label{prop:braiding-W} 
The braiding on $W_{i,m}^\epsilon$ is given by:

$$
c_{i,m}^\epsilon(w_\ell\otimes w_r) = \epsilon\, \xi^{2i((m-i)-(r+\ell))}w_{-r}\otimes w_{\ell+2r}.
$$
\end{proposition}

\begin{proof}
It follows by a direct calculation applying the braiding's formula in \S \ref{subsec:YD-Nichols},
the coaction formula \eqref{eq:coaction-W} and
Lemmas \ref{lem:coaction-e} and \ref{lem:coaction-chi}. Indeed,  
\begin{eqnarray*}
c_{i,m}^\epsilon (w_{\ell}\otimes w_{r}) &=& \sum_s \left((\chi_{m,m-2i} e_{\ell,s})\hit_{i^\epsilon} w_{r}\right)\otimes w_{s}\\
&\stackrel{s=\ell+2r}{=}& \left(\epsilon\xi^{2i(r-\ell)}\chi_{m,m-2i}\hit_{i^\epsilon} w_{-r}\right) \otimes  w_{\ell+2r}\\
&=& \epsilon\xi^{2i(m-i-r-\ell)} w_{-r}\otimes w_{\ell+2r}.
\end{eqnarray*}
\end{proof}

We end this section with the classification of all simple objects in ${}_{K_{n}}^{K_{n}}\mathcal{YD}$. For 
the explicit description of the structure and the braiding of these, see \S \ref{sec:structureVandU} and \S \ref{sec:structureW}

\begin{theorem}\label{thm:class-simple-objects}
Every simple Yetter-Drinfeld module $V$ over $K_{n}$ is isomorphic to one of the module described above, that is, for 
$\epsilon =\pm 1$ and $i,j,m,t \in \Zn$:
\begin{itemize}
 \item if $\dim V = 1$, then $V\simeq V_{i,m}^\epsilon$;
 \item if $\dim V = 2$, then $V\simeq U_{i,j,m,t}$ with $i\neq j$ or $t\neq m-2i$;
 \item if $\dim V = n$, then $V\simeq W_{i,m}^\epsilon$.
\end{itemize}
\end{theorem}

\begin{proof}
From Subsections  \S \ref{sec:structureVandU} and \S \ref{sec:structureW}, we know that 
the modules $V_{i,m}^\epsilon$, $U_{i,m,t}$, $U_{i,j,m,t}$ and $W_{i,m}^\epsilon$ 
with $\epsilon =\pm 1$, $i,j,m,t \in \Zn$ and $i\neq j$, $t\neq m-2i$ constitute a family of pairwise
non-isomorphic simple modules. Then, by counting dimensions we get from \eqref{eq:dimension-simple-kNB} and 
Corollary \ref{cor:amountWs} that 
\begin{align*}
&\sum_{\epsilon, i,m\in \Zn} (\dim V_{i,m}^\epsilon)^{2} +  \sum_{i,m,t\in \Zn\atop t\neq m-2i} (\dim U_{i,i,m,t})^{2} + 
\sum_{i,j,m,t\in \Zn\atop i\neq j,\ t\neq m-2i} (\dim U_{i,j, m,t})^{2} + \sum_{\epsilon, i,m\in \Zn} (\dim W_{i,m}^\epsilon)^{2} =\\
& = 2n^2\cdot 1+ \frac{1}{2} n^2(n-1)\cdot 4 + \frac{1}{2} n^3(n-1)\cdot 4 + 2n^{2}\cdot n^{2} = 4n^{4} = \dim D(K_{n}). 
\end{align*}
Thus, by the Artin-Wedderburn theorem this family provides a full set of pairwise
non-isomorphic simple objects in ${}_{K_{n}}^{K_{n}}\mathcal{YD}$.
\end{proof}

\section{The fusion ring of ${}_{K_{n}}^{K_{n}}\mathcal{YD}$}\label{sec:fusion-ring}

For the reader convenience we recall some notation and results from previous sections.  
We fix a primitive an odd integer $n\ge 3$ and $\xi$ a primitive $n$-th root of one.  
The Hopf algebra $K_{n}=\kk^N\rtimes_\beta \kk Q$ with $N=C_{n}\times C_{n}$, $Q=C_{2}$ and
 $\beta_{x}(a^{i}b^{j},a^{k}b^{\ell}) = \xi^{i\ell - jk}$ for all $i,j,k,\ell \in \Zn$, has basis 
{ $\mathcal{B}=\{p_{i,j},f_{i,j}: i,j\in\Zn\}$, where $\{p_{i,j}\}$} 
is the dual basis in $\kk^N$ of the basis $\{a^ib^j: i,j\in\Zn\}$ of $\kk N$ and 
{$f_{i,j}=p_{i,j}\wx$.} 
By Proposition \ref{prop:comatrix-basis}, the subspace
$C=\kk^N \wx$ is isomorphic as a coalgebra to $\mathcal{M}_n(\kk)^*$; 
the comatrix basis $(e_{k\ell})_{k,\ell\in\Zn } $ is given by 
$$
e_{k,\ell}=\sum_s \xi^{-2s(k+\ell)} f_{s+k-\ell,s-k+\ell},
$$
that is,  $\De (e_{k,\ell}) = \sum_r e_{k,r}\ot e_{r,\ell}$ and $\varepsilon(e_{k,\ell}) = \delta_{k,\ell}$ for all $k,\ell \in \Zn$.

\bigskip
The simple Yetter-Drinfeld modules over $B$ are given by the following families, here $\ep = \pm 1$ and
$i,j,m,t \in \Zn$:
\smallskip
\begin{enumerate}
\item[$\mathbf{(V^\epsilon_{i,m})}$.] $V_{i,m}^\ep=\kk\{v\}$ where the 
action is given by $f\act v = f(i,i)v$, $\wx\act v = \ep v$, 
and the coaction  by $\delta(v)=\chi_{m,-2i+m}\otimes v$.  
These modules are irreducible and pairwise non-isomorphic.

\smallskip
\item[$\mathbf{(U_{i,j,m,t})}$.] $U_{i,j,m,t}=\kk\{u_1, u_2\}$ 
where the action is given by $f\act u_1= f(i,j)u_1$, $f\cdot u_2=f(j,i) u_2$, $\wx\cdot u_1=u_2$,
$\wx\cdot u_2=u_1$
and the coaction by $\delta(u_1)=\chi_{m,t}\ot u_1$, $\de(u_2)=\chi_{t+2i,m-2j}\ot u_2$.  
\smallskip

Two of these modules, say $U_{i,j,m,t}$ and $U_{i',j',m',t'}$ are isomorphic 
if and only if $(i',j',m',t')\in\{(i,j,m,t),(j,i,2i+t,-2i+m)\}$. 
We also remark that if $i\not=j$, then it is impossible to have both $m=t+2i$ and $t=m-2j$.
These modules are reducible if and only if $i=j$ and $t=m-2i$. 
If this happens then $U_{i,i,m,-2i+m}\simeq V^+_{i,m}\oplus V^-_{i,m}$.  
The dual action (with respect to the basis $\mathcal{B}=\{e_{ij},f_{ij}: i,j\in\Zn\}$) is given by 
$p_{a,b}\dact u_1 = \chi_{m,t}(a,b) u_1=\xi^{ma+tb} u_1$, $p_{a,b}\dact u_2=\chi_{t+2i,m-2j}(a,b) u_2 = 
\xi^{(2i+t)a+(-2j+m)b} u_2$, and 
$f_{a,b}\dact u_{k}=0$ for $k=1,2$.

\smallskip
\item[$\mathbf{(W_0)}$.] $W_0=\kk\{w_0,\ldots, w_{n-1}\}$ 
where the action is given by $f\act w_r=f(2r,-2r) w_r$,
$\wx\act w_r = w_{-r}$, and the coaction by $\de(w_r)=\sum_k e_{rk}\ot w_k$.  

\smallskip
\item[$\mathbf{(W_{i,m}^\ep)}$.] $W_{i,m}^\ep=V_{i,m}^\ep\otimes W_0$.  
If we identify $w_r$ with $v\ot w_r$, then the action is given by 
$f\act_{i^\ep} w_r = \chi(i+2r,i-2r) w_r$, $\wx\act_{i^\ep} w_r = \ep\, w_{-r}$, and the coaction  by 
$\de_{i,m}^\ep(w_r)=\sum_k \chi_{m,m-2i}e_{rk}\ot w_k$.  These modules are irreducible and pairwise non-isomorphic.
\end{enumerate}

\bigskip

\subsection{Fusion rules}\label{subsec:fusion-rules}
Below we compute the fusion rules of ${}_{K_{n}}^{K_{n}}\mathcal{YD}$. 
Since $W_{i,m}^\ep=V_{i,m}^\ep\otimes W_0$ and the category is braided,
it suffices to compute the fusion rules between the simple modules of dimension less or equal than two and $W_{0}$.  

\bigskip
\noindent
$\mathbf{V_{i_1,m_1}^{\ep_1}\ot V_{i_2,m_2}^{\ep_2}:}$
It is fairly obvious that 
$$
V_{i_1,m_1}^{\ep_1}\ot V_{i_2,m_2}^{\ep_2}\simeq V_{i_1+i_2,m_1+m_2}^{\ep_1 \ep_2}.
$$

\bigskip
\noindent
$\mathbf{U_{i_1,j_1,m_1,t_1}\ot U_{i_2,j_2,m_2,t_2}:}$
Denote the generators of the first tensor factor by 
$\bu{1}{1},\bu{2}{1}$ and the generators of the second tensor factor by $\bu{1}{2},\bu{2}{2}$.  
This tensor product decomposes, as a Yetter-Drinfeld module, into the direct sum 
$\kk\{\bu{1}{1}\ot \bu{1}{2},\bu{2}{1}\ot\bu{2}{2}\}\oplus \kk\{\bu{1}{1}\ot \bu{2}{2},\bu{2}{1}\ot\bu{1}{2}\}$.  
Direct comparison shows that the first summand is isomorphic to $U_{i_1+i_2,j_1+j_2,m_1+m_2,t_1+t_1}$ 
(via the isomorphism induced by $\bu{1}{1}\ot\bu{1}{2}\mapsto u_1, \bu{2}{1}\ot\bu{2}{2}\mapsto u_2$)
and the second summand is isomorphic to 
$U_{i_1+j_2,j_1+i_2,m_1+2i_2+t_2,t_1-2j_2+m_2}$
(via the isomorphism induced by $\bu{1}{1}\ot\bu{2}{2}\mapsto u_1, \bu{2}{1}\ot\bu{1}{2}\mapsto u_2$).
In conclusion,
$$
U_{i_1,j_1,m_1,t_1}\ot U_{i_2,j_2,m_2,t_2} \simeq 
U_{i_1+i_2,j_1+j_2,m_1+m_2,t_1+t_1}\oplus U_{i_1+j_2,j_1+i_2,m_1+2i_2+t_2,t_1-2j_2+m_2}.
$$

\bigskip
\noindent
$\mathbf{V_{i_1,m_1}^\ep \otimes U_{i_2,j_2,m_2,t_2}:}$
In a similar fashion as above we also see that 
$$
V_{i_1,m_1}^\ep \otimes U_{i_2,j_2,m_2,t_2} \simeq U_{i_1+i_2, i_1+j_2,m_1+m_2,-2i_1+m_1+t_2}.
$$  
The isomorphism is given by $v\ot u_1\mapsto u_1$ and $v\ot u_2\mapsto \ep u_2$.   

\bigskip
\noindent
$\mathbf{W_0\ot W_0:}$
 We first compute the action dual to the coaction 
with respect to the basis $\mathcal{B}$.  Let $\langle -,-\rangle$ denote the 
standard pairing with respect to $\mathcal{B}$, i.e., 
for $z=\sum_{a,b} (\lambda_{a,b} p_{ab}+\mu_{ab}f_{ab})$ 
we have that $\langle p_{ab},z\rangle =\lambda_{ab}$ and $\langle f_{ab},z\rangle =\mu_{ab}$.
Then 
\begin{eqnarray*}
\langle p_{ab}, e_{pk}e_{qm}\rangle &=& \left\langle p_{ab}, 
\sum_{c,d} \xi^{-2c(p+k)-2d(q+m)} f_{c+p-k,c-p+k}f_{d+q-m,d-q+m}\right\rangle \\
&=& \begin{cases} 
\xi^{-(a+b)(p+k+q+m)} &,\mbox{ if } 2(p-k)=a-b,\ 2(q-m)=-a+b \\
0 &, \mbox{ otherwise }
\end{cases}\\
&=& 
\begin{cases} \xi^{-2(a+b)(p+q)} &,\mbox{ if } k=p-\frac{a-b}{2}, m=q+\frac{a-b}{2} \\
0 &, \mbox{ otherwise} 
\end{cases}.
\end{eqnarray*}
This second equality is obtained by observing that $f_{c+p-k,c-p+k}f_{d+q-m,d-q+m}$ is
$p_{c+p-k,c-p+k}$ when $c+p-k=d-q+m$ and $c-p+k=d+q-m$ and is $0$ otherwise.
Hence $f_{ab}\dact (w_p\ot w_q)=0$ and 
$$
p_{ab}\dact (w_p\ot w_q)=\xi^{-2(a+b)(p+q)} w_{p-\frac{a-b}{2}}\ot w_{q+\frac{a-b}{2}}.
$$
Now set $\bv{j}{k}:=w_{k+j}\ot w_{k-j}$.  Then, by the above, we have that
\begin{align*}
p_{ab}\dact\bv{j}{k} &= p_{ab}\dact(w_{k+j}\ot w_{k-j})
= \xi^{2(a+b)(2k)} w_{k+j-\frac{a-b}{2}}\ot w_{k-j+\frac{a-b}{2}} = \xi^{-4k(a+b)} \bv{j-\frac{a-b}{2}}{k} .
\end{align*}
Also note that 
\begin{eqnarray*}
\chi\act \bv{j}{k}&=& \chi\act(w_{k+j}\ot w_{k-j}) = \chi(4k,-4k)\bv{j}{k},
\end{eqnarray*}
and
\begin{align*}
\wx\act \bv{j}{k}&= \wx\act (w_{k+j}\ot w_{k-j}) 
= \sum_{a,b,c,d} \xi^{ad-bc} ((p_{ab}\wx)\act w_{k+j}) \ot ((p_{cd}\wx)\act w_{k-j})\\
&= \sum_{a,b,c,d} \xi^{ad-bc} (p_{ab}\act w_{-k-j})\ot (p_{cd}\act w_{-k+j}) 
= w_{-k-j}\ot w_{-k+j}= \bv{-j}{-k}.
\end{align*}
The last equality 
follows from the observation that
in order to get a non-zero summand we need to have $a=-2(k+j), b=2(k+j), c=2(-k+j)$, and $d=-2(-k+j)$.

\bigskip
Now we introduce the elements
$$
\by{r}{k}=\sum_{j} \xi^{jr}\bv{j}{k} \qquad\qquad{\text{for all }r,k\in\Zn.}
$$
Then, the following identities hold
\begin{eqnarray*}
\wx\act \by{r}{k} &=& \by{-r}{-k} \\
\chi\act \by{r}{k} &=& \chi(4k,-4k)\by{r}{k}\\
f_{ab}\dact \by{r}{k} &=& 0 \\
p_{ab}\dact \by{r}{k} &=& \sum_{j} p_{ab}\dact (\xi^{jr}\bv{j}{k} ) = \sum_{j} \xi^{jr-4k(a+b)}\bv{j-\frac{a-b}{2}}{k} \\
&=& \sum_{j} \xi^{(j-\frac{a-b}{2})r}\xi^{(\frac{a-b}{2})r}\xi^{-4k(a+b)}\bv{j-\frac{a-b}{2}}{k}\\&=& \xi^{(\frac{a-b}{2})r-4k(a+b)}
\by{r}{k}\\
&=& \xi^{(-4k+\frac{1}{2}r)a+(-4k-\frac{1}{2}r)b} \by{r}{k}.
\end{eqnarray*}
From this we see that $\kk \by{0}{0}$ and $\kk\{\by{r}{k},\by{-r}{-k}\}$, $(r,k)\not=(0,0)$ 
are Yetter-Drinfeld modules over $K_{n}$. Moreover, 
$\kk \by{0}{0}\simeq V_{0,0}^+$ 
and
$\kk\{\by{r}{k},\by{-r}{-k}\}\simeq U_{4k,-4k, -4k+\frac{1}{2}r,-4k-\frac{1}{2} r}$ via the isomorphism given by 
$\by{r}{k}\mapsto u_1$, $\by{-r}{-k}\mapsto u_2$ (it is also isomorphic to $U_{-4k,4k,4k-\frac{1}{2} r, 4k+\frac{1}{2} r}$ 
via the isomorphism that switches $u_1$ and $u_2$).  
Note that, since $(r,k)\not=(0,0)$ we cannot simultaneously have 
$4k=-4k$ and $-4k+\frac{1}{2}r=-4k-\frac{1}{2} r$ and hence these Yetter-Drinfeld modules are irreducible.
Denote by $\mathcal{Z}_{n}$ the set of 
isomorphism classes in $\Zn\times \Zn$ given by the relation $(r,k)\sim \pm (r,k)$.
Then,
$$
W_0\ot W_0 \simeq V_{0,0}^+\oplus \bigoplus_{[r,k]\in \mathcal{Z}_{n}\atop (r,k)\neq (0,0)}  
U_{-4k,4k,4k-\frac{1}{2} r, 4k+\frac{1}{2} r}
$$

\bigskip
\noindent
$\mathbf{U_{i,j,m,t}\ot W_0:}$ 
Lastly, we analyse the decomposition of $U_{i,j,m,t}\ot W_0$.  We will prove that 
$$
U_{i,j,m,t}\ot W_0\simeq W^+_{\frac{i+j}{2},\frac{m+t}{2}+i}\oplus W^-_{\frac{i+j}{2},\frac{m+t}{2}+i}
$$
by exhibiting an explicit isomorphism 
$$
\varphi\colon U_{i',i',m',m'-2i'}\ot W_0\to U_{i,j,m,t}\ot W_0,
$$ 
where 
$$
i'=\frac{i+j}{2}\qquad \mbox{ and }\qquad m'=\frac{m+t+2i}{2}.
$$ 
The two-dimensional module
$U_{i',i',m',m'-2i'}$ is not simple, in fact 
$U_{i',i',m',m'-2i'} \simeq V^{+}_{i',m'}\oplus V^{-}_{i',m'}$.
Then, it follows that 
$$
U_{i',j',m',t'}\ot W_{0}\simeq 
\big(V^{+}_{i',m'}\oplus V^{-}_{i',m'}\big)\ot W_{0} \simeq
\big(V^{+}_{i',m'}\ot W_{0}\big) \oplus \big(V^{-}_{i',m'}\ot W_{0}\big) \simeq W_{i',m'}^{+} \oplus W_{i',m'}^{-}
$$

Set $D=\frac{i-j}{4}$ and 
$M=(m-t)-(m'-t')=m-t-i-j$.
Then
this isomorphism $\varphi$ is given by 
\begin{eqnarray*}
\varphi(u'_1\ot w_r)&=&\xi^{-rM}u_1\ot w_{r-D},\qquad \qquad\forall\ r\in \Zn,\\
\varphi(u'_2\ot w_r)&=&\xi^{rM-2D(i+j)}u_2\ot w_{r+D}.
\end{eqnarray*} 

\begin{remark}{\em It is clear that $U_{i,j,m,t}\ot W_0$ is isomorphic as an $\kk^N\wx$-comodule to 
$W_0\oplus W_0$ and therefore by Theorem \ref{thm-13} it must be isomorphic to some 
$W_{i_1,m_1}^{\ep_1}\oplus W_{i_2,m_2}^{\ep_2}$ 
as a Yetter-Drinfeld module.  
Analysis somewhat simpler to what follows can then be use to establish that $\ep_1\ep_2=-1$, 
$i_1=i_2=\frac{i+j}{2}$, $m_1=m_2=\frac{m+t+2i}{2}$.}
\end{remark}

Before we establish that $\varphi$ is an isomorphism of Yetter-Drinfeld modules, 
we analyse the structure of $U_{i,j,m,t}\otimes W_0$ in more detail.
First we compute the action $*$ dual to the coaction. Note that
$$
\langle f_{pq}| e_{rs}\rangle=
\begin{cases} 
\xi^{-(p+q)(r+s)} &,\mbox{ if }s=r-\frac{p-q}{2}\\ 0 &,\mbox{ otherwise}
\end{cases};
$$ 
and hence for any character $\chi$ we have that 
$$
\langle f_{pq} | \chi\, e_{rs}\rangle=
\begin{cases} \chi(p,q)\xi^{-(p+q)(r+s)} &,\mbox{ if }s=r-\frac{p-q}{2}\\ 0 &,\mbox{ otherwise}
\end{cases}.
$$  
As $\chi_{m,t}(p,q)=\xi^{mp+tq}=\xi^{(p+q)\frac{m+t}{2}+(p-q)\frac{m-t}{2}}$,
the dual action in $U_{i,j,m,t}\ot W_0$ is given by

\begin{eqnarray*}
f_{pq}*(u_1\ot w_r) &=& \sum_s \langle f_{pq}|\chi_{m,t}\, e_{rs}\rangle u_1\ot w_s \\
&=& \chi_{m,t}(p,q)\xi^{-(p+q)(2r-\frac{p-q}{2})}u_1\ot w_{r-\frac{p-q}{2}} \\
&=& \xi^{-(p+q)(2r-\frac{p-q}{2}-\frac{m+t}{2})+\frac{p-q}{2}(m-t)} u_1\ot w_{r-\frac{p-q}{2}},
\end{eqnarray*}
and
\begin{eqnarray*}
f_{pq}*(u_2\ot w_r) &=& \sum_s \langle f_{pq}|\chi_{t+2i+m,m-2j}\,e_{rs}\rangle u_2\ot w_s \\
&=& \chi_{t+2i,m-2j}(p,q)\, \xi^{-(p+q)(2r-\frac{p-q}{2})}u_2\ot w_{r-\frac{p-q}{2}} \\
&=& \xi^{-(p+q)(2r-\frac{p-q}{2}-\frac{m+t}{2}-(i-j))+\frac{p-q}{2}(-m+t+2i+2j)}u_2\ot w_{r-\frac{p-q}{2}},
\end{eqnarray*}
for all $p,q,r\in \Zn$.
Hence
\begin{eqnarray*}
f_{pq}*\varphi(u'_1\ot w_r)&=& \xi^{-Mr} f_{pq}*(u_1\ot w_{r-D})\\
&=& \xi^{-Mr-(p+q)(2(r-D)-\frac{p-q}{2}-\frac{m+t}{2})+\frac{p-q}{2}(m-t)}u_1\ot w_{r-D-\frac{p-q}{2}}. 
\end{eqnarray*}
On the other hand:
\begin{eqnarray*}
\varphi(f_{pq}*' (u'_1\ot w_r))&=& 
\xi^{-(p+q)(2r-\frac{p-q}{2}-\frac{m'+t'}{2})+\frac{p-q}{2}(m'-t')} \varphi(u'_1\ot w_{r-\frac{p-q}{2}}) \\
&=& \xi^{-M(r-\frac{p-q}{2})-(p+q)(2r-\frac{p-q}{2}-\frac{m'+t'}{2})+\frac{p-q}{2}(m'-t')} u_1\ot w_{r-D-\frac{p-q}{2}}. 
\end{eqnarray*}
We conclude that the two expressions are equal by observing that 
$-2D-\frac{m+t}{2}=-\frac{m'+t'}{2}$ and $m-t=M+m'-t'$. Similarly, we also get that
\begin{eqnarray*}
f_{pq}*\varphi(u'_2\ot w_r)&=& \xi^{Mr-2D(i+j)} f_{pq}*(u_2\ot w_{r+D})\\
&=& \xi^{Mr-2D(i+j)-(p+q)(2(r+D)-\frac{p-q}{2}-\frac{m+t}{2}-(i-j))+\frac{p-q}{2}(-m+t+2i+2j)} u_2\ot w_{r+D-\frac{p-q}{2}}. 
\end{eqnarray*}
and
\begin{align*}
&\varphi(f_{pq}*'(u'_2\ot w_r))= 
\xi^{-(p+q)(2r-\frac{p-q}{2}-\frac{m'+t'}{2}-(i'-j'))+\frac{p-q}{2}(m-t+2i+2j)} \varphi(u'_2\ot w_{r-\frac{p-q}{2}})\\
&\quad = 
\xi^{M(r-\frac{p-q}{2})-2D(i+j)-(p+q)(2r-\frac{p-q}{2}-\frac{m'+t'}{2}-(i'-j'))+
\frac{p-q}{2}(-m'+t'+2i'+2j')} u_2\ot w_{r+D-\frac{p-q}{2}}. 
\end{align*}
We get that the two expressions are equal by noting that 
$2D-\frac{m+t}{2}-(i-j)=-\frac{m'+t'}{2}=-\frac{m'+t'}{2}-(i'-j')$ and that $-m+t+2i+2j=-M-m'+t'+2i'+2j'$.
As the isomorphism $\varphi$ preserves the dual
action, it follows that it is a comodule map.

\bigbreak
We next address the $\widehat{x}$-action.  Since

\begin{align*}
\widehat{x}\act (u_1\otimes w_r) &= \sum_{a,b,c,d} \xi^{ad-bc} p_{ab}\widehat{x}\act u_1 \ot p_{cd}\widehat{x}\act w_r 
= \sum_{a,b,c,d} \xi^{ad-bc} p_{ab}\act u_2 \ot p_{cd}\act w_{-r}\\ 
& = \xi^{2r(i+j)}u_2\ot w_{-r},\qquad\text{ and }\\ 
%
%
\widehat{x}\act (u_2\otimes w_r) &= \sum_{a,b,c,d}  \xi^{ad-bc} p_{ab}\widehat{x}\act u_2 \ot p_{cd}\widehat{x}\act w_r 
= \sum_{a,b,c,d} \xi^{ad-bc} p_{ab}\act u_1 \ot p_{cd}\act w_{-r} \\
&= \xi^{2r(i+j)}u_1\ot w_{-r}, 
\end{align*}
we have that 
$$
\widehat{x}\act \varphi(u'_1\otimes w_r) = \xi^{-Mr} \widehat{x}\act(u_1\otimes w_{r-D}) 
= \xi^{-Mr+2(r-D)(i+j)} u_2\ot w_{-r+D}
$$
is equal to
$$
\varphi(\widehat{x}\act' (u'_1\otimes w_r)) = \xi^{2r(i'+j')} \varphi(u'_2\ot w_{-r}) \\
= \xi^{M(-r)-2D(i+j)+2r(i'+j')} u_2\ot w_{-r+D},
$$
as $i+j=i'+j'$.
Similarly,
$$
\widehat{x}\act \varphi(u'_2\otimes w_r) = \xi^{Mr-2D(i+j)} \widehat{x}\act(u_2\otimes w_{r+D}) \\
= \xi^{Mr-2D(i+j)+2(r+D)(i+j)} u_1\ot w_{-r-D},
$$
is equal to
$$
\varphi(\widehat{x}\act' (u'_2\otimes w_r)) = \xi^{2r(i+j)} \varphi(u'_1\otimes w_{-r}) \\
= \xi^{2r(i+j)-M(-r)} u_1\otimes w_{-r-D}.
$$

\bigbreak
We now conclude the proof by the following computations: for any character $\chi \in \kk^{N} $ we have
\begin{eqnarray*}
\chi\act \varphi(u'_1\ot w_r) &=& \xi^{-Mr}\chi\act (u_1\ot w_{r-D})\\
&=& \xi^{-Mr}\chi(i+2(r-D),j-2(r-D)) u_1\ot w_{r-D} \\
&=& \xi^{-Mr}\chi(i'+2r,i'-2r) u_1\ot w_{r-D} \\
&=& \varphi(\chi\act' (u_1\ot w_r)),
\end{eqnarray*}
\begin{eqnarray*}
\chi\act \varphi(u'_2\ot w_r) &=& \xi^{Mr-2D(i+j)}\chi\act (u_2\ot w_{r+D}) \\
&=& \xi^{Mr-2D(i+j)}\chi(j+2(r+D),i-2(r+D)) u_2\ot w_{r+D} \\
&=& \xi^{Mr-2D(i+j)}\chi(i'+2r,i'-2r) u_2\ot w_{r+D} \\
&=& \varphi(\chi\act' (u_2\ot w_r)).
\end{eqnarray*}
As the characters span linearly $\kk^{N}$,  $\varphi$ is a module map.

\bigskip
We end this section with the description of the fusion ring of ${}_{K_{n}}^{K_{n}}\mathcal{YD}$.

\begin{theorem}\label{thm:fusion-ring}
 The fusion ring $\mathcal{F}$ of ${}_{K_{n}}^{K_{n}}\mathcal{YD}$ is the commutative ring generated
 by the elements $v_{i,m}^{\ep}$, $u_{i,j,m,t}$, $w^{\ep}_{i,m}$ with 
 $\ep= \pm 1$, $i,j,m,t \in \Zn$ and $t\neq m-2i$ when $i=j$, satisfying the following relations: (set $w^{+}_{0,0} = w_{0}$)
 \smallskip
 \begin{eqnarray*}
  v_{i_1,m_1}^{\ep_1} v_{i_2,m_2}^{\ep_2} &= & v_{i_1+i_2,m_1+m_2}^{\ep_1 \ep_2},\\
  v_{i,m}^{\ep} w_{0} &=& w^{\ep}_{i,m},\\ 
  v_{i_1,m_1}^\ep u_{i_2,j_2,m_2,t_2} &=& u_{i_1+i_2, i_1+j_2,m_1+m_2,-2i_1+m_1+t_2},\\
  u_{i_1,j_1,m_1,t_1} u_{i_2,j_2,m_2,t_2} & = &u_{i_1+i_2,j_1+j_2,m_1+m_2,t_1+t_1}+ u_{i_1+j_2,j_1+i_2,m_1+2i_2+t_2,t_1-2j_2+m_2},\\
  u_{i,j,m,t} w_0 &=& w^+_{\frac{i+j}{2},\frac{2i+m+t}{2}} + w^-_{\frac{i+j}{2},\frac{2i+m+t}{2}},\\
w_0\ot w_0  &=&  v_{0,0}^+ + \sum_{[r,k]\in \mathcal{Z}_{n}\atop (r,k)\neq (0,0)}  u_{-4k,4k,4k-\frac{1}{2} r, 4k+\frac{1}{2} r},
 \end{eqnarray*}
  where $\mathcal{Z}_{n}$ is the set of isomorphism classes in $\Zn\times \Zn$ given by the relation $(r,k)\sim \pm (r,k)$.
  \qed
\end{theorem}

\section{Nichols algebras}\label{sec:Nichols-alg}

{ In this last section we compute the Nichols algebras associated with some 
modules in ${}_{K_{n}}^{K_{n}}\mathcal{YD}$. The families of Yetter-Drinfeld modules 
$\{V^{\ep}_{i,m}\}_{\ep, i,m}$ and $\{U_{i,j,m,t}\}_{i,j,m,t \in \Zn}$ consist 
of braided vector spaces of diagonal type, thus their Nichols algebras 
can be completely described by the work of Heckenberger \cite{He} and Angiono \cite{An}.
On the other hand, 
the braided vector spaces $W^{\ep}_{i,m}$ turn out to be
of rack type and isomorphic to braided vector spaces associated with 
the dihedral rack and a constant cocycle, i.e. a conjugacy 
class of an involution in the dihedral group $\mathbb{D}_{n}$ and a
one-dimensional representation. In the particular case for $n=3$, 
we determine all finite-dimensional Nichols algebras over simple modules.
Here, the well-known
Fomin-Kirillov algebra $\mathcal{E}_{3}$ appears as a Nichols algebra over $K_{3}$.
We include in this section
the presentation of the finite-dimensional Nichols algebras over simple
modules,
which includes one $12$-dimensional Nichols algebra which is not
isomorphic to $\mathcal{E}_{3}$.
As a consequence, we obtain new Hopf algebras of dimension $216$ by bosonization.}

\subsection{Nichols algebras of sums of one-dimensional modules $V^\epsilon_{i,m}$}\label{subsec:Nichols-V}\,

For $i,m \in \Zn$ and $\epsilon = \pm 1$, the Yetter-Drinfeld modules $V^\epsilon_{i,m}$ 
are one-dimensional vector spaces generated by an element $v_{i}:= v_{i,m}^{\epsilon}$.  
Their structure and braiding is given in Subsection \ref{sec:structureVandU} item $\mathbf{(V^\epsilon_{i,m})}$. 
{
From the very definition, we get the following proposition. 

\begin{proposition}\label{prop:Nichols-one-dim}
Let $i,m \in \Zn$ and $\epsilon = \pm 1$ and set $\ell =  \ord\big(\xi^{i(m-i)}\big)$.
Then 
\begin{align*}
 \B(V^\epsilon_{i,m}) \simeq  
 \begin{cases}
  \begin{array}{ll}
   \kk[v_{i}] & \text{ if }\ell = 1;\\
   \kk[v_{i}]/(v_{i}^{\ell}) & \text{ otherwise}.
  \end{array}
 \end{cases}
\end{align*}
\qed
\end{proposition}
}

A Yetter-Drinfeld module $V = \bigoplus_{(\epsilon, i,m)\in I} V^\epsilon_{i,m}$ 
given by a direct sum of finitely many one-dimensional 
simple modules is a braided vector space of diagonal type with basis $\{v^{\epsilon}_{i,m}\}_{(\epsilon,i,m)\in I}$. The braiding is given by 
$$
c(v^{\epsilon}_{i,m}\otimes v^{\eta}_{j,\ell}) =\xi^{2j(m-i)} v^{\eta}_{j,\ell}\otimes v^{\epsilon}_{i,m}
$$
for all triples $(\epsilon,i,m)$ and $(\eta,j,\ell)$ in $I$. 
{
In case $V = V^\epsilon_{i,m} \oplus V^\eta_{j,\ell}$, the braiding matrix is
$$
\mathbf{q} = 
\left(
\begin{matrix}
 \xi^{2i(m-i)} & \xi^{2j(m-i)}\\
  \xi^{2i(\ell-j)}& \xi^{2j(\ell-j)}
\end{matrix}
\right).
$$
Since $n$ is odd, by \cite[Theorem 15.3.3]{HS}
we have that $\B(V)$ is finite-dimensional if
and only if
$i(m-i)\neq 0 \neq j(\ell -j) \in \Zn$
and the generalized Dynkin diagram

\bigbreak
$$
\begin{picture}(20,10)(0,0)
\put(0,0){$\circ$}
\put(0,10){\tiny{$\xi^{2i(m-i)}$}}
\put(4.5,3){\line(1,0){60}} 
\put(64,10){\tiny{$\xi^{2j(\ell-j)}$}}
\put(64,0){$\circ$}
\put(15,-10){\tiny{$\xi^{2j(m-i) + 2i(\ell-j)}$}}
\end{picture}
$$
\vspace*{0.5cm}

\noindent is isomorphic to one of the rows 
$1,\ 2,\ 4,\ 6,\ 7,\ 11,\ 12$ or $17$ of \cite[Table 15.1]{HS}. For example, the diagram is isomorphic
to the one in
row $1$ if $\xi^{2j(m-i) + 2i(\ell-j)} = 1$, that is ${j(m-i)= - i(\ell-j)} \in \Zn$. In this 
case, there is no edge between the vertices and 
the Nichols algebra is isomorphic to a 
\textit{quantum linear space}
$$
\B(V) \simeq \kk\{ x_{i},x_{j}\ :\ 
x_{i}^{n_{i}},x_{j}^{n_{j}}, x_{i}x_{j} - \xi^{2j(m-i)}x_{j}x_{i}\},
$$
where $n_{i} = \ord\big(\xi^{i(m-i)}\big)$ 
and $n_{j} = \ord\big(\xi^{j(\ell-j)}\big)$. Here
we wrote $x_{i} = v_{i,m}^{\epsilon}$ and 
$x_{j} = v_{j,\ell}^{\eta}$ to simplify the
presentation.

On the other hand, the diagram is isomorphic to 
the one in row $2$ if $i(m-i) = j(\ell-j)$ and 
$-i(m-i) = jm+i\ell -2ij$ in $\Zn$. 
In such a case, the braiding is of Cartan type $A_{2}$. As above, write $x_{i} = v_{i,m}^{\epsilon}$ and 
$x_{j} = v_{j,\ell}^{\eta}$. Set  
$\ad(x)(y) = [x,y]_{c} = 
xy - m\circ c(x\ot y)$ for $x,y\in T(V)$ 
and denote $x_{ij} = \ad(x_{i})(x_{j})$.
Then, 
\begin{equation}\label{eq:nichols-A2}
\B(V) \simeq \kk\{ x_{i},x_{j}\ :\ 
x_{i}^{N},\, x_{j}^{N},\, x_{ij}^{N},\, \ad^{2}(x_{i})(x_{j}),\, \ad^{2}(x_{j})(x_{i})\}
\simeq u_{\xi^{i(m-i)}}(\mathfrak{sl}_{3})^{+},
 \end{equation}
where $N = \ord\big(\xi^{i(m-i)}\big)$.

As one may deduce from the examples above, the 
presentation of the Nichols algebras depends on the 
arithmetics in $\Zn$.
With patience and hard work one may obtain 
the complete list of finite-dimensional Nichols algebras for a fix $n$ and a given rank
by analysing Heckenberger's list of arithmetic root system in \cite{He} and computing the 
presentation following Angiono's result in
\cite{An}.
}

\subsection{Nichols algebras of sums of two-dimensional modules $U_{i,j,m,t}$}\label{subsec:Nichols-U}\,

For $i,j,m,t \in \Zn$ the Yetter-Drinfeld modules $U_{i,j,m,t}$ 
are two-dimensional vector spaces spanned by the elements $u_{1}, u_{2}$. 
Their structure and braiding is given in Subsection \ref{sec:structureVandU} item $\mathbf{(U_{i,j,m,t})}$. 
In particular, the braided vector spaces $U_{i,j,m,t}$ are of diagonal type; the braiding matrix 
and the corresponding generalized Dynkin diagram are as follows:

\begin{align*}
\mathbf{q} = 
\left(
\begin{matrix}
 \xi^{mi+tj} & \xi^{ti+mj}\\
  \xi^{ti+mj + 2(i^{2}-j^{2})}& \xi^{mi+tj}
\end{matrix}
\right) && 
\begin{picture}(20,10)(0,0)
\put(0,0){$\circ$}
\put(0,10){\tiny{$\xi^{mi + tj}$}}
\put(4.5,3){\line(1,0){60}} 
\put(64,10){\tiny{$\xi^{mi + tj}$}}
\put(64,0){$\circ$}
\put(15,-10){\tiny{$\xi^{ti + mj + 2(i^{2}-j^{2})}$}}
\end{picture} \qquad\qquad
\end{align*}
%
{

Then, $\B(U_{i,j,m,t})$
is finite-dimensional if and only if $mi+tj\neq 0$
and $(i+j)(m +t)= 2(j^{2}-i^{2})$ in $\Zn$, as the diagram above 
must be isomorphic to the one in row 2 of 
\cite[Table 15.1]{HS}. In such a case, the braided vector space is of Cartan type $A_{2}$ and 
the presentation is the one given in \eqref{eq:nichols-A2}.
We state the result below.

\begin{proposition}\label{prop:Nichols-two-dim-finite}
Let $i,j,m,t \in \Zn$. Then 
$\B(U_{i,j,m,t})$
is finite-dimensional if and only if $mi+tj\neq 0$
and $(i+j)(m +t)= 2(j^{2}-i^{2})$ in $\Zn$.
In such a case, 
$$
\B(U_{i,j,m,t}) \simeq \kk\{ x_{i},x_{j}\ :\ 
x_{i}^{N},\, x_{j}^{N},\, x_{ij}^{N},\, \ad^{2}(x_{i})(x_{j}),\, \ad^{2}(x_{j})(x_{i})\}
\simeq u_{q}(\mathfrak{sl}_{3})^{+},
$$
where $N = \ord\big(\xi^{mi+tj}\big)$ and $q=\xi^{\frac{mi+tj}{2}}$.
\qed
\end{proposition}

\begin{remark}
If $i=j$ and $t=m-2i$, then $U_{i,i,m,m-2i} \simeq V_{i,m}^{+}\oplus V_{i,m}^{-}$ and the generalized Dynkin diagram equals 

$$
\begin{picture}(20,10)(0,0)
\put(0,0){$\circ$}
\put(0,10){\tiny{$\xi^{2i(m-i)}$}}
\put(4.5,3){\line(1,0){60}} 
\put(64,10){\tiny{$\xi^{2i(m-i)}$}}
\put(64,0){$\circ$}
\put(15,-10){\tiny{$\xi^{4i(m-i)}$}}
\end{picture} 
$$

\bigbreak
\noindent Then $\B(V_{i,m}^{+}\oplus V_{i,m}^{-})$
is finite-dimensional if and only if $i(m-i)\neq 0$
and $6i(m-i)= 0$ in $\Zn$. In such a case, 
$\B(V_{i,m}^{+}\oplus V_{i,m}^{-}) \simeq u_{\xi^{i(m-i)}}(\mathfrak{sl}_{3})^{+}$.
\end{remark}

Now we analyze the Nichols algebra of a braided
vectos space given by a finite sum of simple
two-dimensional modules. 

\begin{theorem}\label{nichols-sum-U}
Let $I$ be a finite subset of $\Zn^{4}$
and 
$V= \oplus_{(i,j,m,t) \in I} U_{i,j,m,t}$
be a braided vector space given by the direct 
sum of simple two-dimensional modules. Then 
$\B(V)$ is finite if and only if
\begin{enumerate}
 \item[$(a)$] $mi+tj\neq 0$
and $(i+j)(m +t)= 2(j^{2}-i^{2})$ in $\Zn$ for all 
$(i,j,m,t) \in I$, 
 \item[$(b)$] $0= mk + t\ell + pi + sj $ and $0=pj+si+tk+2ik+m\ell -2j\ell$
 in $\Zn$ for all  $(i,j,m,t), (k,\ell,p,s) \in I$.
\end{enumerate}
In such a case, $\B(V)$ is the braided tensor product of Nichols algebras
isomorphic to $u_{q}(\mathfrak{sl}_{3})^{+}$ with
$q=\xi^{\frac{mi+tj}{2}}$ for all $(i,j,m,t) \in I$.
\end{theorem}

\begin{proof}
Assume $\dim \B(V)$ is finite. Then, $\dim \B(U_{i,j,m,t})$ must be finite for every $(i,j,m,t) \in I$. So, by Proposition
\ref{prop:Nichols-two-dim-finite}, we must have that $mi+tj\neq 0$
and $(i+j)(m +t)= 2(j^{2}-i^{2})$ in $\Zn$ for all $(i,j,m,t) \in I$; this gives the conditions in $(a)$.
Now take two summands $U_{i,j,m,t}$ and $U_{k,\ell,p,s}$ in $V$ with bases
$\{u_{1},u_{2}\}$ and $\{u'_{1},u'_{2}\}$, respectively. Since the braiding on $U_{i,j,m,t}\oplus U_{k,\ell,p,s}$ is of 
diagonal type to analyse the dimension of the Nichols algebra on this sum one has to check if the 
two $A_{2}$-type diagrams of these modules are connected. To do this, we compute the braiding 
between vectors of these bases. For example,
\begin{align*}
 c(u_{1}\ot u'_{1}) &= \chi_{m,t}\cdot u'_{1}\ot u_{1} = \chi_{m,t}(a^{k}b^{\ell})\, u'_{1}\ot u_{1}=\xi^{mk+t\ell}\, u'_{1}\ot u_{1},\\
 c(u'_{1}\ot u_{1}) &= \chi_{p,s}\cdot u_{1}\ot u'_{1} = \chi_{p,s}(a^{i}b^{j})\, u'_{1}\ot u_{1}=\xi^{pi+sj}\, u'_{1}\ot u_{1}.
 \end{align*}
Then, the vertex corresponding to $u_{1}$ is connected to the one corresponding to $u_{1}'$ if and only if
$0\neq  mk + t\ell + pi + sj \in \Zn$. Performing the same computation for the elements $u_{2}$ and $u_{2}$ yields 
\begin{align*}
 c(u_{2}\ot u'_{2}) &= \chi_{t+2i,m-2j}\cdot u'_{2}\ot u_{2} = \chi_{t+2i,m-2j}(a^{\ell}b^{k})\, u'_{2}\ot u_{2}=\xi^{\ell (t+2i)+k(m-2j)}\, u'_{2}\ot u_{2}, \\
 c(u'_{2}\ot u_{2}) &= \chi_{s+2k,p-2\ell}\cdot u_{2}\ot u'_{2} = \chi_{s+2k,p-2\ell}(a^{j}b^{i})\, u'_{2}\ot u_{2}=\xi^{j(s+2k)+i(p-2\ell)}\, u'_{2}\ot u_{2}.
 \end{align*}
So, the vertex corresponding to $u_{2}$ is connected to the one corresponding to $u_{2}'$ if and only if
$0\neq  \ell (t+2i)+k(m-2j) +j(s+2k)+i(p-2\ell) = \ell t + km +js + ip  \in \Zn$, which is exactly the same condition 
on the vertices corresponding to $u_{1}$ and $u'_{1}$. Hence, $u_{1}$ is connected to $u'_{1}$ if and only if 
$u_{2}$ is connected to $u'_{2}$. Since in \cite[Table 3]{He} there are no squares, 
one must have that $0 = \ell t + km +js + ip  \in \Zn$, which is the first condition on $(b)$.
The second condition follows by analizing the connection between the vertices 
$u_{1}$ and $u'_{2}$, and $u_{2}$ with $u'_{1}$. As above, the former pair of vertices is connected
if and only if the latter is. Hence, both $A_{2}$-type diagrams must be desconnected. As this holds
for each pair of modules, one concludes that the generalized Dynkin diagram corresponding to $V$ is 
the union of all the generalized Dynkin diagrams corresponding the summands $U_{i,j,m,t}$. Thus,
the Nichols algebra $\B(V)$ is isomorphic to the braided tensor product of Nichols algebras
$\B(U_{i,j,m,t})$, that is $\B(V) \simeq \underline{\bigotimes}_{(i,j,m,t)\in I} \B(U_{i,j,m,t})$. The last assertion 
of the statement follows from Proposition \ref{prop:Nichols-two-dim-finite}.
\end{proof}

}

\subsection{Nichols algebras of the n-dimensional modules $W^\epsilon_{i,m}$}\,

For $i,m\in\Zn$ and $\epsilon \in \{\pm 1\}$, let $W^{\epsilon}_{i,m} = \Bbbk\{w_{0},\ldots,w_{n-1}\}$ 
be the braided vector space with the structure 
described in Subsection \ref{sec:structureW} item $\mathbf{(W^\epsilon_{i,m})}$. In particular,
the braiding is given by 
\begin{equation}\label{eq:braiding-W}
c^{\epsilon}_{i,m}(w_{\ell}\ot w_{r}) = 
\epsilon\, \xi^{2i(m-i-r-\ell)}\, w_{-r}\ot w_{\ell+2r}
\qquad\text{ for all }r,\ell \in \Zn.
\end{equation}

From \S \ref{subsec:YD-Nichols} and Proposition \ref{prop:braiding-W} follows at once
that $\dim \B(W^{+1}_{i,m})$ is infinite whenever $i=0$ or $i=m$, since in such a case 
$c^{+}_{i,m}(w_{0}\ot w_{0}) = w_{0}\ot w_{0}$.

\smallbreak
For the remaining cases, we will make use of the
theory of braided vector spaces associated with 
set-theoretical solutions to the braid equation.
For a detailed exposition see \cite{AG}.

\subsubsection{Set-theoretical solutions to the braid equation}
Let X be a non-empty set and let 
$s : X \times X \to X \times X$ be a bijection. 
We say that $s$ is a \textit{set-theoretical solution}
to the braid equation (or solution for short) if 
$$
(s\times \id)(\id\times s)(s\times \id) =
(\id\times s)(s\times \id)(\id\times s)
$$
as maps on $X\times X\times X$.
Clearly, the identity map and the flip 
$\tau: X \times X \to X \times X$, $\tau(x,y) = (y,x)$
for all $x,y\in X$ are solutions. 
A \textit{braided set} is then a pair $(X,s)$ where $X$ is a non-empty set and $s$ is a solution.
If $(X,s)$ braided set,
there is an action of 
the braid group $\mathbb{B}_{n}$ on $X^{n}$: the standard generators $\sigma_{i}$ act by $s$ on the
$i, i + 1$ entries.

\smallbreak 
 Let $(X,s)$ be a braided set and let 
 $f , g : X \to \Fun (X,X)$ be given by
$$
s(x,y) = (g_{x}(y),f_{y}(x))
\qquad\text{ for all }x,y\in X.
$$
The solution (or the braided set) is
called \textit{non-degenerate} if the images 
of $f$ and $g$ are bijections.

In our case, the braidings $c^{\epsilon}_{i,m}$ are related to
the set theoretical solution $(\Zn, s)$, where 
\begin{equation}\label{eq:set-theor-example}
s: \Zn \times \Zn \to \Zn\times \Zn,\qquad 
s(\ell,r) = (-r, \ell+2r)\qquad \text{ for all }\ell,r\in \Zn
\end{equation}
Here $g_{\ell}(r) = -r$ and $f_{r}(\ell) = \ell + 2r$
for all $\ell, r\in \Zn$. As $n$ is assumed to be odd,
the braided set $(\Zn, s)$ is non-degenerate.

The scalars 
$F_{m,i,\ell,r}^{\epsilon} = \epsilon\, \xi^{2i(m-i-r-\ell)}$ appearing in the braiding  $c^{\epsilon}_{i,m}$ are codified in a notion similar to a $2$-cocycle.
Let $X$ be a finite set, $s:X\times X \to X\times X$
a bijection and 
$F : X \times X \to \CC^{\times}$ a function.
Denote by $\CC X$ the vector space with basis $X$ and 
define $s^{F}: \CC X \otimes \CC X \to \CC X\ot \CC X$
by 
\begin{equation}\label{eq:set-braiding}
s^{F}(x\ot y) = F_{x,y}\, s(x,y) =F_{x,y}\, g_{x}(y)\ot f_{y}(x)
\end{equation}

\begin{lemma}\cite[Lemma 5.7]{AG}
$s^{F}$ is a solution of the braid equation if and 
only if $(X,s)$ is a braided set and 
\begin{equation}\label{eq:cocycle}
F_{x,y}F_{f_{y}(x),z}F_{g_{x}(y),g_{f_{y}(x)}(z)} =
F_{y,z}F_{x,g_{y}(z)}F_{f_{g_{y}(z)}(x), f_{z}(y)}
\qquad\text{ for all }x,y,z\in X.
\end{equation}
\qed
\end{lemma}

\begin{definition}\cite[Definition 5.8]{AG}
Let $(X,s)$
be a non-degenerate solution and $F : X \times X \to \CC^{\times}$ a function such that 
\eqref{eq:cocycle} holds. 
We say that the braided vector space    
$(\CC X, s^{F})$ is of
\textit{set-theoretical type}. 
\end{definition}

Directly from the lemma above we have that for all
$i,m\in \Zn$ and $\epsilon \in \{\pm 1\}$ the 
function $F_{i,m}^{\epsilon}: \Zn\times \Zn \to \CC^{\times}$ given by $F_{i,m,\ell,r}^{\epsilon} = 
\epsilon\, \xi^{2i(m-i-\ell-r)}$ satisfies 
\eqref{eq:cocycle}; it may also be checked directly. In conclusion, our 
braided vector spaces $(W^{\epsilon}_{i,m}, c^{\epsilon}_{i,m})$ are of set-theoretical type, with the solution $(\Zn, s)$, where 
$s(\ell,r) = (-r, \ell+2r)$ for all $\ell,r\in \Zn$.

\subsubsection{Racks}
Any set-theoretical solution can be described
in terms of racks. A \textit{rack} is a pair 
$(X,\triangleright)$ where $X$ is a non-empty set and 
$\triangleright: X\times X \to X$ is a function
such that $x\triangleright -:X\to X$ is a bijection 
for all $x\in X$ and 
$x\triangleright(y\triangleright z) = 
(x\triangleright y) \triangleright (x\triangleright z)$ for all $x,y,z\in X$.
The archetypical example of a rack is a union of conjugacy classes in a group $G$ where 
the map $\triangleright$ is given by the conjugation, i.e. $x\triangleright y = xyx^{-1}$.
For example, for $G=\mathbb{D}_{n} = \langle g,h\ |\ g^{2}=1=h^{n},\ ghg=h^{n-1}\rangle$
the dihedral group of order $2n$, the conjugacy class $\Oc_{gh}$ of the involution
$gh$ is a rack, with $\Oc_{gh}=\{g^{2i+1}h\, : \, 0\leq i\leq n-1\}$ and 
$$
(g^{2j+1}h)\triangleright (g^{2i+1}h) = (g^{2j+1}h)(g^{2i+1}h)(g^{2j+1}h)^{-1} =g^{2(2j-i)+1}h
$$
For $n$ odd this rack has size $n$, and for $n$ even has size $\frac{n}{2}$. 
In terms of racks, we may describe $\Oc_{gh}$ is a simpler way by writing 
$g^{2i+1}h\ = x_{i}$ for all $0\leq i\leq n-1$. Then 
\begin{equation}\label{eq:dihedral-rack}
\Oc_{gh} =:\D_{n}=\{x_{i}\, :\, 0\leq i\leq n-1\}\qquad
\text{ and }\qquad
x_{j}\triangleright x_{i} = x_{2j-i}\quad\text{ for all }0\leq i,j\leq n-1.
\end{equation}

Racks give rise to set-theoretical solutions to the braid equation. 
Assume $X$ is a non-empty set and let  $\triangleright: X \times X\to X$ be a function. Let
$c:X\times X\to X\times X$ be the function given by 
$c(x,y) = (x\triangleright y, x)$ for all $x,y\in X$. 
Then $c$ is a solution if and only if $(X,\triangleright)$ is a rack.

\smallbreak
From any non-degenerate braided set $(X,s)$ with $s(x,y) = (g_{x}(y), f_{y}(x))$ for $x,y\in X$ 
one may construct a rack $(X,\triangleright)$ which yields another solution, called the \textit{derived solution} of $s$.

\begin{proposition}\label{prop:set-th-rack}
 Let $s$ be a non-degenerate solution and define
 $$
 x\triangleright y = f_{x}(g_{f^{-1}_{y}(x)}(y))
 $$
 If $c:X\times X\to X\times X$ is given by 
 $c(x,y) = (x\triangleright y,x)$, then $c$ is a solution; we call it the derived solution of $s$.
Moreover, the solutions $s$ and $c$ are equivalent and $(X,\triangleright)$ is a rack.
\qed
\end{proposition}

Any rack and a $2$-cocycle on it give rise to a braided vector space.
Let $(X,\triangleright)$ be a rack and $q:X\times X \to \CC^{\times}$ be a function with notation
$q_{xy}:= q(x,y)$ for all $i,j\in X$
such that 
\begin{equation}\label{eq:cond-rack-cocycle}
 q_{x,y\triangleright z}q_{y,z} = q_{x\triangleright y, x \triangleright z}q_{x,z}\qquad\text{ for all }x,y,z\in X 
\end{equation}
Then the vector space $V=\CC X$ with basis the elements of $X$ is a braided vector space 
with braiding $c^{q}:\CC X\ot \CC X \to \CC X\ot \CC X$ given by 
$$
c^{q}(x\ot y) = q_{x,y}\, x\triangleright y \ot x\qquad \text{ for all }x,y\in X.
$$
We denote this braided vector space by $(\CC X, c^{q})$ and the corresponding Nichols
algebra by $\B(X,c^{q})$. The function $q:X\times X \to \CC^{\times}$ satisfying \eqref{eq:cond-rack-cocycle}
is called a \textit{rack 2-cocycle}.

\subsubsection{t-equivalence between braided vector spaces}
There is a relation between braided vector spaces weaker than
isomorphism but useful enough to deal with Nichols algebras.

\begin{definition}\cite[Definition 5.10]{AG}
We say that two braided vector spaces $(V,c)$ and 
$(W,d)$ are 
\textit{t-equivalent} if there is a collection of linear isomorphisms $U^{n}: V^{\ot n} \to W^{\ot n}$
intertwining the corresponding representations of the braid group $\mathbb{B}_{n}$, for all $n\geq 2$.
The collection $(U^{n})_{n\geq 2}$ is called a \textit{t-equivalence}.
\end{definition}

\smallbreak
\begin{remark}\label{rmk:t-equiv}\cite[Example 5.11]{AG}
{\em Let $(\CC X ,s^{F})$ be a braided vector space 
of set-theoretical type and
$(X, c)$ be the derived solution. 
Set $q_{xy} = F_{f^{-1}_{y}(x),y}$ for all $x,y\in X$.
If $q_{f_{z}(x),f_{z}(y)} = q_{xy}$ for all $x,y,z \in X$, then the braided vector spaces 
$(\CC X ,s^{F})$ and $(\CC X ,c^{q}))$ are t-equivalent.}
\end{remark}

\smallbreak
In our example, the braided vector space
$(W^\epsilon_{i,m},c^\epsilon_{i,m})$ can be described using the set-theoretical solution 
$(\Zn,s^{F})$ where 
$s(w_{\ell},w_{r}) = (w_{-r},w_{\ell+2r})$
and $F= F_{i,m,\ell,r}^{\epsilon} = 
\epsilon\, \xi^{2i(m-i-\ell-r)}$. In particular, $g_{\ell}(r) = -r$ and $f_{r}(\ell) = \ell +2r$ for all $\ell, r\in \Zn$
The corresponding derived solution has rack structure $\ell \triangleright r = 2\ell -r $, since
$$
\ell \triangleright r  = f_{\ell}(g_{f^{-1}_{r}(\ell)}(r)) = f_{\ell}(-r) = -r+2\ell  \qquad\text{ for all }\ell,r\in \Zn. 
$$
Hence, $(\Zn,\triangleright) =\D_{n}$ is the dihedral rack. With respect to the cocycle we have 
$q_{\ell,r} = \epsilon\, \xi^{2i(m-i-(\ell-r))}$:
$$
q_{\ell,r} = F_{f^{-1}_{r}(\ell),r} = F_{\ell-2r,r}=\epsilon\, \xi^{2i(m-i-(\ell-2r+r))}=
\epsilon\, \xi^{2i(m-i-(\ell-r))}.
$$
In conclusion, $(W^\epsilon_{i,m},c^\epsilon_{i,m})$ is t-equivalent to the braided vector space
$(\CC \D_{n}, d^\epsilon_{i,m})$ with $\CC \D_{n} =\CC\{x_{\ell}:\ 0\leq \ell \leq n-1\}$ and 
braiding $c^{q} = d^\epsilon_{i,m}$ given by
$$
d^\epsilon_{i,m}(x_{\ell}\otimes x_{r}) = \epsilon\, \xi^{2i(m-i-(\ell-r))} x_{2\ell -r}\ot x_{\ell}\qquad \text{ for all }\ell, r \in \Zn.
$$

\begin{lemma}\label{lem:t-equiv-graded-iso}\cite[Lemma  6.1]{AG}
If $(V,c)$ and $(W,d)$ are t-equivalent braided vector spaces, then the corresponding Nichols algebras 
$\B(V)$ and $\B(W)$ are isomorphic as
graded vector spaces. In particular, one has finite dimension, resp. finite GK-dimension,
if and only if the other one has.
\qed
\end{lemma}

\bigbreak
As a consequence of the lemma above, we have the following:

\begin{corollary}\label{cor:Wim-iso-gvs-Dn}
The Nichols algebras $\B(W^\epsilon_{i,m},c^\epsilon_{i,m})$ are isomorphic
as graded vector spaces to the Nichols 
algebras $\B(\D_{n}, d^\epsilon_{i,m})$. 
\qed
\end{corollary}

As a consequence of the corollary above,  
$\B(W^\epsilon_{i,m},c^\epsilon_{i,m})$ has the same
(Gelfand-Kirillov) dimension as $\B(\D_{n}, d^\epsilon_{i,m})$. 
In case $n$ is prime, these
dimensions are know due to a recent result of Heckenberger, 
Mehir and Vendramin.
The following theorem is a direct consequence of \cite[Theorem 1.6]{HMV}.

\begin{theorem}\label{thm:hecken-mehir-vendra}
Let $n$ be an odd prime. Then  $\B(\D_{n}, d^\epsilon_{i,m})$ is finite-dimensional 
if and only if $n=3$ and there exists a basis $\{y_{k}\}_{k\in \mathbb{Z}_{n}}$
of $\CC\D_{n}$
such that $d^\epsilon_{i,m}(y_{\ell}\ot y_{r}) = - y_{2\ell -r}\ot y_{r}$.
\qed
\end{theorem}

\begin{remark}\label{rmk:notation-braiding}{\em
 For $i=0$, all braided vector spaces $(\D_{n}, d^\epsilon_{0,m})$ coincide. For simplicity, 
we write $d_{0} = d^{-}_{0,m}$ for the braiding corresponding to the 
parameters $i=0$ and $\epsilon = -1$.}
 \end{remark}

\begin{remark}{\em
The braided vector space 
$(\CC \D_{n}, d_{0})$ may be
realized as a Yetter-Drinfeld module over 
the dihedral group $\mathbb{D}_{n}$. 
The braiding is given by 
$$
d_{0}(x_{\ell}\otimes x_{r}) =  - x_{2\ell -r}\ot x_{\ell}\qquad \text{ for all }\ell, r \in \Zn.
$$
The corresponding object is given in group-theoretical
terms by the simple Yetter-Drinfeld module
$M(\Oc_{g},\sgn)$ associated with the conjugacy class
$\Oc_{g}$ of $g$ and the character of the centralizer
$C_{\mathbb{D}_{n}}(g) =\langle g\rangle$ given by the sign representation, i.e. $\sgn(g) = -1$.
In conclusion, $\B(W^-_{0,m},c^\epsilon_{0,m}) $ is isomorphic as graded vector space 
to $\B(\Oc_{g},\sgn)$.}
\end{remark}

The Nichols algebras over $\mathbb{D}_{n}$
were intensively studied and up to a
possible exception, they are all 
infinite-dimensional. The following theorem
extend the results of \cite[Theorem 3.1]{AF}.

\begin{theorem}\label{thm:Nichols-Dn} \cite[Theorem 4.8]{AHS},
\cite[Theorem 1.6]{HMV}.
Assume $n\geq 5$ is odd.
All Nichols algebra over 
$\mathbb{D}_{n}$ are infinite-dimensional with 
the possible exception of $\B(\Oc_{g}, \sgn)$, up to isomofphism, when $n$ is not prime. 
\qed
\end{theorem}

\bigbreak

\subsubsection*{The Nichols algebras $\B(W^{-}_{i,m},c^{-}_{i,m})$ for $n=3$ and $\epsilon = -1$.}\

\bigbreak
Note that, for $n=3$ one has that $\mathbb{D}_{3} = \mathbb{S}_{3}$. 
Since all finite-dimensional Nichols algebras over $\mathbb{S}_{3}$ are known,
we can characterize all finite-dimensional
Nichols algebras over $K_{3}$ thanks to the description of
M. Gra\~na \cite{G} who studied Nichols algebras of low
dimension.

\bigbreak

\noindent
\textbf{Case $i=0$:}  
The Nichols algebra $\B(\D_{3}, d_{0})$ associated with the braided vector space 
$(\CC \D_{3}, d_{0})$ is isomorphic to the well-known \textit{Fomin-Kirillov algebra} $\mathcal{E}_{3}$.
Indeed, $\CC \D_{3} = \{x_{0},x_{1},x_{2}\}$ and 
$$
d_{0}(x_{i}\otimes x_{j}) =  - x_{k}\ot x_{j}\qquad \text{ for }i,j,k \text{ all distinct.}
$$
Its Nichols algebra has dimension $12$, top degree $4$ and Hilbert series 
$\mathcal{H}(t) = t^{4} + 3t^{3} + 4t^{2} + 3t + 1$.
It is the quadratic algebra generated by the elements $x_{0},x_{1},x_{2}$ satisfying the relations
\begin{eqnarray}
x_{i}^{2}= 0 \quad \text{ for all }i\nonumber \\
x_{0}x_{1} + x_{1}x_{2} + x_{2}x_{0} = 0\label{eq:rel-Fomin-Kirillov} \\
x_{0}x_{2} + x_{2}x_{1} + x_{1}x_{0} = 0\nonumber
\end{eqnarray}

By the previous discussion we know that $\B(W^{-}_{0,m},c^{-}_{0,m})$ is t-equivalent to 
$\mathcal{E}_{3} $. The following theorem shows that they are indeed isomorphic.

\begin{theorem}\label{thm:pres-Wom}
Let $m\in \mathbb{Z}_{3}$. Then $\B(W^{-}_{0,m},c^{-}_{0,m})\simeq\mathcal{E}_{3} $; in particular,
it admits the presentation \eqref{eq:rel-Fomin-Kirillov} and the following one as the algebra
generated by the elements $w_{0},w_{1},w_{2}$ satisfying the relations
\begin{eqnarray}
w_{0}^{2}= 0, \qquad w_{1}w_{2}=0, \qquad w_{2}w_{1}=0,\nonumber \\
w_{0}w_{1} + w_{1}w_{0} + w_{2}^{2} = 0,\label{eq:presentation-W0m} \\
w_{0}w_{2} + w_{2}w_{0} + w_{1}^{2} = 0.\nonumber
\end{eqnarray}
\end{theorem}

\begin{proof}
By the remark above, the braided vector spaces are t-equivalent. We show here that 
moreover, 
$(W^{-}_{0,m},c^{-}_{0,m})$ is isomorphic to $(\CC \D_{3}, d_{0})$ as braided
vector space. Indeed,
the isomorphism is given by the following linear map 
\begin{equation}\label{eq:def-linear-trans}
\varphi: \CC \D_{3} \to W^{-}_{0,m},\qquad \varphi(x_{k}) = 
w_{0} + \xi^{k} w_{1}+\xi^{2k} w_{2},\qquad\text{ for all }0\leq k\leq 2,   
 \end{equation}
which is a morphism between braided vector spaces, since
\begin{align*}
c_{0,m}^{-}(\varphi(x_{\ell})\ot \varphi(x_{r})) & =
c_{0,m}^{-}\big((w_{0} + \xi^{\ell} w_{1}+\xi^{2\ell} w_{2})\ot (w_{0} + \xi^{r} w_{1}+\xi^{2r} w_{2})\big)\\
& = c_{0,m}^{-}(w_{0}\ot w_{0}) +  \xi^{r} c_{0,m}^{-}(w_{0}\ot w_{1}) +\xi^{2r} c_{0,m}^{-}(w_{0}\ot w_{2}) +\\
& \ + \xi^{\ell} c_{0,m}^{-}(w_{1}\ot w_{0}) +  \xi^{\ell+r} c_{0,m}^{-}(w_{1}\ot w_{1}) +\xi^{\ell+2r} c_{0,m}^{-}(w_{1}\ot w_{2}) +\\
& \ +\xi^{2\ell}c_{0,m}^{-}(w_{2}\ot w_{0}) +  \xi^{2\ell+r} c_{0,m}^{-}(w_{2}\ot w_{1}) +\xi^{2\ell+2r} c_{0,m}^{-}(w_{2}\ot w_{2})\\
& = -w_{0}\ot w_{0} -  \xi^{r} w_{2}\ot w_{2} -\xi^{2r} w_{1}\ot w_{1} +\\
& \ - \xi^{\ell} w_{0}\ot w_{1} -  \xi^{\ell+r} w_{2}\ot w_{0} -\xi^{\ell+2r} w_{1}\ot w_{2} +\\
& \ -\xi^{2\ell} w_{0}\ot w_{2} -  \xi^{2\ell+r} w_{2}\ot w_{1} -\xi^{2\ell+2r} w_{1}\ot w_{0}\\
&= - w_{0}\ot \varphi(x_{\ell}) - \xi^{\ell + r} w_{2}\ot \varphi(x_{\ell}) - \xi^{2\ell + 2r} w_{1}\ot \varphi(x_{\ell}) \\
&= - \varphi(x_{-r+2\ell}) \ot \varphi(x_{\ell}) = (\varphi\ot\varphi) (-x_{2\ell-r} \ot x_{\ell} ) = (\varphi\ot\varphi) d_{0} (x_{\ell}\ot x_{r})
\end{align*}
for all $\ell, r, m\in \mathbb{Z}_{3}$. 

Now consider the quadratic approximation 
$\hat{\B}_{2}(W^{-}_{0,m},c^{-}_{0,m}) = T(W^{-}_{0,m},c^{-}_{0,m}) / \mathcal{J}_{2}$. 
It is the
quadratic algebra presented by the elements
$\{w_{0},w_{1},w_{2}\}$ satisfying the relations
\begin{eqnarray}
w_{0}^{2}= 0, \qquad w_{1}w_{2}=0, \qquad w_{2}w_{1}=0,\nonumber \\
w_{0}w_{1} + w_{1}w_{0} + w_{2}^{2} = 0\nonumber \\
w_{0}w_{2} + w_{2}w_{0} + w_{1}^{2} = 0\nonumber
\end{eqnarray}
Using the quadratic relations one may show that the homogeneous component 
of degree $4$ of $\hat{\B}_{2}(W^{-}_{0,m})$ is linearly spanned by the element
$w_{0}w_{1}w_{0}w_{2}$ and the algebra vanishes in degree 5. 
Then by \cite[Theorem 6.4]{AG}, we have that $\hat{\B}_{2}(W^{-}_{0,m}) = \B(W^{-}_{0,m})$ and we obtain 
another presentation of this Nichols algebras.
\end{proof}

\bigbreak

\noindent
\textbf{Case $i=m$:}  
Write $d_{i} = d^{-}_{i,i}$ for $i\in \mathbb{Z}_{3}$. Then $(\CC \D_{3},d_{i})$ has braiding 
$$
d_{i}(x_{\ell}\ot x_{r}) =  - \xi^{-2i(\ell-r)} x_{2\ell -r}\ot x_{\ell}\qquad \text{ for all }\ell, r \in \mathbb{Z}_{3}.
$$
By \cite[Lemma 3.8]{G}, we have that  $(\CC \D_{3},d_{i})$ is of group-type and
performing the change of basis $y_{k}=\xi^{-ik} x_{k}$ yields
$$
d_{i}(y_{\ell}\ot y_{r}) =  -  y_{2\ell -r}\ot y_{\ell}\qquad \text{ for all }\ell, r \in \mathbb{Z}_{3}.
$$
Indeed, 
\begin{align*}
 d_{i}(y_{\ell}\ot y_{r}) & = \xi^{-i(\ell + r)}d_{i}(x_{\ell}\ot x_{r}) = -\xi^{-i(\ell + r)} \xi^{-2i(\ell-r)} x_{2\ell -r}\ot x_{\ell} = 
 -\xi^{-i\ell} \xi^{-i(2\ell-r)} x_{2\ell -r}\ot x_{\ell} \\
 & = -y_{2\ell -r}\ot y_{\ell}
\end{align*}
Hence, all braided vector spaces $(W^{-}_{i,i},c^{-}_{i,i})$ with $i\in \mathbb{Z}_{3}$ are t-equivalent to
$(\CC \D_{3}, d_{0})$. As a consequence, $\dim \B(W^{-}_{i,i},c^{-}_{i,i}) = 12$ for all 
$i\in \mathbb{Z}_{3}$. 

\bigbreak
Based on the fact above, we introduce the following notion. 
\bigbreak

\begin{definition}\label{def:twist-equiv}

We say that two braided vector spaces $(\CC X,s^{F})$ and $(\CC X,s^{G})$ of set-theoretical type are \textit{twist-equivalent} if the 
braided vector spaces corresponding to the derived solutions are \textit{twist-equivalent}, see \cite{AFGaV}, \cite{V}. 
Explicitly, write 
$s(x,y) = (g_{x}(y),f_{y}(x))$
and set 
$x\triangleright y = f_{x}(g_{f^{-1}_{y}(x)}(y))$ for all $x,y\in X$. Then 
$(\CC X,s^{F})$ and $(\CC X,s^{G})$ are twist-equivalent if there exists a map 
$\varphi: X\times X \to \CC^{\times}$ 
such that 
$$
\varphi(x,z)\varphi(x\triangleright y,x\triangleright z)\varphi(x\triangleright(y\triangleright z),x)\varphi(y\triangleright z,y)=
\varphi(y,z)\varphi(x,y\triangleright z)
\varphi(x\triangleright (y\triangleright z), x\triangleright y)\varphi(x\triangleright z, x)
$$
for all $x,y,z\in X$ and 
$$
\varphi(x,y) F_{f^{-1}_{y}(x),y} = \varphi(x\triangleright y,x) G_{f^{-1}_{y}(x),y}\qquad\text{ for all }x,y\in X.
$$
In such a case, we write $G = F^{\varphi}$. One may re-write the equation above in 
terms of the bijective maps $f , g : X \to \Fun (X,X)$ and 
replacing $x$ by $f_{y}(x)$ as
\begin{equation}\label{eq:twist-equiv}
\varphi(f_{y}(x),y) F_{x,y} = \varphi(f_{x}(g_{x}(y)),x) G_{x,y}\qquad\text{ for all }x,y\in X.
 \end{equation}
\end{definition}

From the very definition, Lemma \ref{lem:t-equiv-graded-iso} and the results in \cite{AFGaV}, we have the following: if $(\CC X,s^{F})$ and $(\CC X,s^{G})$ are twist-equivalent
as braided vector spaces, then their 
Nichols algebras $\B(\CC X,s^{F})$ and $\B(\CC X,s^{G})$ are isomorphic as graded
vector spaces.

\bigbreak

In our examples, taking 
$(W^{-}_{k,k},c_{k,k}^{-})=(\CC X,s^{F})$ and $(W^{-}_{i,i},c_{i,i}^{-})=(\CC X,s^{G})$ we have that $X=\mathbb{Z}_{3}$, 
$g_{\ell}(r) = -r$, $f_{r}(\ell) = \ell + 2r$, $F= F_{k,k,\ell,r}^{-} = -\xi^{-2k(\ell + r)}$, 
$G= F_{i,i,\ell,r}^{-} = -\xi^{-2i(\ell + r)}$, $\ell \triangleright r   = -r+2\ell$ for all $\ell, r\in \mathbb{Z}_{3}$ and 
\eqref{eq:twist-equiv} reads

\begin{equation}\label{eq:twist-equiv-example}
\xi^{-2k(\ell + r)} \varphi(\ell+2r,r) = \xi^{-2i(\ell + r)}\varphi(-r+2\ell,\ell) \qquad\text{ for all }x,y\in X.
 \end{equation}

\bigbreak

The map $\varphi_{i,k}: \mathbb{Z}_{3} \times\mathbb{Z}_{3}  \to \CC^{\times}$ given by 
$\varphi_{i,k} (\ell,r) = \xi^{(i-k)(\ell -2r)}$ clearly satisfies 
\eqref{eq:twist-equiv-example} and the cocycle condition above.
Thus, the braided vector spaces $(W^{-}_{i,i},c_{i,i}^{-})$ are twist-equivalent for all $i\in \mathbb{Z}_{3}$.
In particular, the corresponding Nichols algebras are isomorphic as graded vector spaces and we get 
that the top degree of both $\B (W^{-}_{1,1})$ and $\B (W^{-}_{2,2})$ is 4.

\begin{theorem}\label{thm:Nichols-alg-Wii}
{ With the notation above, the algebras $\B (W^{-}_{1,1})$ and $\B (W^{-}_{2,2})$ has the following 
presentation} 
\begin{align*}
\B(W^{-}_{1,1})& =\Bbbk \{w_{0},w_{1},w_{2}\ :\ w_{0}^{2},\ w_{1}w_{2},\ w_{2}w_{1},\ 
\xi^{2}w_{0}w_{2} + w_{2}w_{0} + \xi w_{1}^{2},\ {\xi^{2}w_{0}w_{1} +\xi  w_{1}w_{0} + w_{2}^{2}}\},\\
\B(W^{-}_{2,2})& =\Bbbk \{w_{0},w_{1},w_{2}\ :\ w_{0}^{2},\ w_{1}w_{2},\ w_{2}w_{1},\ 
\xi w_{0}w_{2} + w_{2}w_{0} + \xi^{2} w_{1}^{2},\ \xi w_{0}w_{1} + \xi^{2}w_{1}w_{0} + w_{2}^{2}\}.
\end{align*}
{ Moreover, these are isomorphic as graded algebras.} 
\end{theorem}

\begin{proof}
Computing the kernels of degree 2 of the quantum symmetrizer associated with the braided 
vector spaces $W^{-}_{i,i}$ for $i=1,2$, yield the following presentations of the corresponding quadratic approximations 
$\hat{\B}_{2}(W^{-}_{i,i}) = T(W^{-}_{i,i}) / \mathcal{J}_{2} $
of the Nichols algebras:
\begin{align*}
\hat{\B}_{2}(W^{-}_{1,1})& =\Bbbk \{w_{0},w_{1},w_{2}\ :\ w_{0}^{2},\ w_{1}w_{2},\ w_{2}w_{1},\ 
\xi^{2}w_{0}w_{2} + w_{2}w_{0} + \xi w_{1}^{2},\ {\xi^{2}w_{0}w_{1} +\xi  w_{1}w_{0} + w_{2}^{2}}\}\\
\hat{\B}_{2} (W^{-}_{2,2})& =\Bbbk \{w_{0},w_{1},w_{2}\ :\ w_{0}^{2},\ w_{1}w_{2},\ w_{2}w_{1},\ 
\xi w_{0}w_{2} + w_{2}w_{0} + \xi^{2} w_{1}^{2},\ \xi w_{0}w_{1} + \xi^{2}w_{1}w_{0} + w_{2}^{2}\}
\end{align*}
As in the proof of Theorem \ref{thm:pres-Wom}, a quick check using the quadratic relations 
gives that the homogeneous component 
of degree $4$ of both $\hat{\B}_{2}(W^{-}_{i,i})$ is linearly spanned by the element
$w_{0}w_{1}w_{0}w_{2}$ and the algebras vanish in degree 5. 
Then by \cite[Theorem 6.4]{AG}, we have that $\hat{\B}_{2}(W^{-}_{i,i}) = \B(W^{-}_{i,i})$ and we obtain 
a presentation of both Nichols algebras.

{
There is a way to get rid of the parameter $\xi$
in the presentations above.
Performing a change of basis in $W^{-}_{i,i}$ suggested by the linear 
transformation in \eqref{eq:def-linear-trans} 
$$
x_{k} = w_{0} + \xi^{k} w_{1}+\xi^{2k} w_{2},\qquad\text{ for all }0\leq k\leq 2,   
$$
one gets a different expression for the braiding:
$$
c_{i,i}^{-}(x_{\ell}\otimes x_{r}) :=  - x_{2\ell - r + i}\ot x_{\ell + i}\qquad \text{ for }i,\ell,r \in \mathbb{Z}_{3},
$$
and consequently another presentation for the Nichols algebras:
\begin{align*}
\B (W^{-}_{1,1})& =\Bbbk \{x_{0},x_{1},x_{2}\ :\ x_{0}x_{1},\ x_{1}x_{2},\ x_{2}x_{0},\ 
x_{0}x_{2} + x_{2}x_{1} + x_{1}x_{0},\ x_{0}^{2} + x_{1}^{2} + x_{2}^{2}\},\\
\B (W^{-}_{2,2})& =\Bbbk \{x_{0},x_{1},x_{2}\ :\ x_{0}x_{2},\ x_{1}x_{0},\ x_{2}x_{1},\ 
x_{0}x_{1} + x_{1}x_{2} + x_{2}x_{0},\ x_{0}^{2} + x_{1}^{2} + x_{2}^{2}\}.
\end{align*}
With this presentations, it is clear that the linear map $\varphi$ sending $x_{k}\mapsto x_{-k}$ interchanges
the presentations of $\B (W^{-}_{1,1})$ and $\B (W^{-}_{2,2})$. In fact, this is an homomorphism
of braided vector spaces $\varphi: W^{-}_{1,1} \to W^{-}_{2,2}$, since 
$(\varphi\ot \varphi)c_{1,1}^{-}(x_{\ell}\ot x_{r})  = -\varphi(x_{2(\ell +r)+1})\ot \varphi(x_{\ell + 1}) = - x_{\ell + r + 2}\ot x_{2\ell + 2} = 
c_{2,2}^{-}(x_{2\ell}\ot x_{2r}) = c_{2,2}^{-}\big(\varphi(x_{\ell})\ot \varphi(x_{r})\big)$. Thus, 
$\B (W^{-}_{1,1})$ and $\B (W^{-}_{2,2})$ are isomorphic as graded algebras, although they are not isomorphic
as objects in ${}_{K_{3}}^{K_{3}}\mathcal{YD}$.}
 \end{proof}

 \begin{remark}{\em
  Using the Majid-Radford product or \textit{bosonization}, one may consider the 
  Hopf algebras $\mathcal{E}_{3}\#K_{3}$, $\B(W^{-}_{1,1})\#K_{3}$ and $\B(W^{-}_{2,2})\#K_{3}$.
  These are $216$-dimensional non-pointed non-semisimple Hopf algebras whose coradical
  is isomorphic to $K_{3}$. Up to our best knowledge, these Hopf algebras were not considered 
  in the literature yet. Observe that using the presentation of the Nichols algebras and that of $K_{3}$, one can
  obtain and present these Hopf algebras by generators and relations. }
 \end{remark}

 \begin{remark}{\em
{
Note that the only difference between  
$\B(W^{-}_{0,m})$, 
$\B(W^{-}_{1,1})$ and $\B(W^{-}_{2,2})$ is the choice of
the $3$rd root of unity. This is clearly seen in the 
description of the braiding and the presentations. Indeed, 
if we write $\B_{\xi^{k}} =
\Bbbk \{w_{0},w_{1},w_{2}\ :\ w_{0}^{2},\ w_{1}w_{2},\ w_{2}w_{1},\ 
\xi^{2k}w_{0}w_{2} + w_{2}w_{0} + \xi^{k} w_{1}^{2},\ \xi^{2k}w_{0}w_{1} +\xi^{k}  w_{1}w_{0} + w_{2}^{2}\}$, then 
$\B_{\xi} = \B(W^{-}_{1,1})$,
$\B_{1} = \B(W^{-}_{0,m})$
and $\B(W^{-}_{2,2}) = \B_{\xi^{2}} $.}
}
\end{remark}

We know that the braided vector spaces $(W_{i,i}^{-}, c_{i,i}^{-})$ are t-equivalent
for all $i\in \mathbb{Z}_{3}$; in particular, the corresponding Nichols algebras
are isomorphic as graded vector spaces. Nevertheless, the 
Nichols algebras $\B_1$ and $\B_\xi$ are 
not isomorphic as algebras by the following theorem.

\begin{theorem}\label{thm:Bi-not-iso-B1}
$\B_1$ and $\B_\xi$ are not isomorphic as algebras.
\end{theorem}

\begin{proof} 
We first prove that $\B_1$ is generated by three generators 
$a,b,c$ such that $a^2=b^2=c^2=0$. 
A quick direct computation shows that 
$a=w_0, b=w_0+\xi w_1+\xi^2 w_2, c=w_0+\xi^2 w_1+\xi w_2$ 
are such generators.

We will now prove that this does not happen for $\B_\xi$, i.e., that no $\B_\xi$ 
cannot be generated by elements $a,b,c$ that satisfy 
$a^2=b^2=c^2=0$.  
Suppose, toward contradiction, that such $a,b,c$ exist.  
Denote by $x_{\ell}$ the homogeneous component of an element $x\in \B_{\xi}$
of degree $\ell$.
Since $\B_\xi$ is a graded algebra with $(\B_\xi)_0=\kk$ we conclude that
\begin{enumerate}
\item $a_0=b_0=c_0=0$,
\item $a_1^2=b_1^2=c_1^2=0$,
\item $a_1,b_1,c_1$ must span $V=(\B_\xi)_1$, 
and since $V$ is $3$-dimensional, this means that 
$a_1, b_1, c_1$ must be linearly independent.
\end{enumerate}
We will prove that the only elements $x\in V$ satisfying 
$x^2=0$ are multiples of $w_0$, thereby arriving at a contradiction.  
Suppose now that $x=\lambda_0 w_0+\lambda_1 w_1+\lambda_2 w_2\in V$ is 
such that $x^2=0$ in $\B_\xi$.  A quick computation shows that this implies that
the element 
\[y=\lambda_0\lambda_1(w_0w_1+w_1w_0)+\lambda_2^2 w_2^2+ 
\lambda_0\lambda_2(w_0w_2+w_2 w_0)+\lambda_1^2 w_1^2\in T(V)
\]
must be a linear combination of elements
\begin{align*}
r_1 &= \xi w_0 w_2 + \xi^2 w_2 w_0 + w_1^2, \\
r_2 &= \xi^2 w_0 w_1 + \xi w_1 w_0 + w_2^2
\end{align*}
in $T(V)$.  Comparing coefficients of $w_1^2$ and $w_2^2$ this can only happen if 
\[
y=\lambda_1^2 r_1 +\lambda_2^2r_2.
\]
But then we must have that $\lambda_0\lambda_2=\xi^2 \lambda_1^2=\xi \lambda_1^2$,
$\lambda_0\lambda_1 = \xi \lambda_2^2=\xi^2 \lambda_2^2$, and hence $\lambda_1=\lambda_2=0$.

\end{proof}

\bigbreak

\noindent
\textbf{Case $i\neq0$, $i\neq m$:}
By Corollary \ref{cor:Wim-iso-gvs-Dn}, we know that 
the Nichols algebras $\B(W^\epsilon_{i,m},c^\epsilon_{i,m})$ are isomorphic
as graded vector spaces to the Nichols 
algebras $\B(\D_{3}, d^\epsilon_{i,m})$. The latter are finite-dimensional 
if and only if 
there exists a basis $\{y_{k}\}_{k\in \mathbb{Z}_{3}}$
of $\CC\D_{3}$
such that $d^\epsilon_{i,m}(y_{\ell}\ot y_{r}) = - y_{2\ell -r}\ot y_{r}$,
by Theorem \ref{thm:hecken-mehir-vendra}. This can only occur only if
$i=0$ or $i=m$.


\begin{thebibliography}{00}  
\bibitem{A} Andruskiewitsch, N.
An Introduction to Nichols Algebras. 
In Quantization, Geometry and Noncommutative Structures in Mathematics and Physics. 
Mathematical Physics Studies, Springer; Alexander Cardona, Pedro Morales, 
Hern\'an Ocampo, Sylvie Paycha, Andr\'es Reyes, eds.  pp. 135--195 (2017).

\bibitem{AA} 
Andruskiewitsch, N.; Angiono, I.
On Nichols algebras over basic Hopf algebras. 
\textit{Math. Z.} \textbf{296} (2020), no. 3-4, 1429--1469.

\bibitem{AAH} 
Andruskiewitsch, N.; Angiono, I.; Heckenberger, I. 
On Nichols algebras of infinite rank with finite Gelfand-Kirillov dimension. 
\textit{Atti Accad. Naz. Lincei Rend. Lincei Mat. Appl.} \textbf{31} (2020), no. 1, 81--101.

\bibitem{AC} Andruskiewitsch, N.; Cuadra, J.
On the structure of (co-Frobenius) Hopf algebras. 
\textit{J. Noncommut. Geom.} \textbf{7} (2013), no. 1, 83--104.

\bibitem{AS} Andruskiewitsch, N.; Schneider, H-J.
On the classification of finite-dimensional pointed Hopf algebras. 
\textit{Ann. of Math.} (2) \textbf{171} (2010), no. 1, 375--417.

\bibitem{AF}
Andruskiewitsch, N.; Fantino, F. 
On pointed Hopf algebras associated with alternating
and dihedral groups. 
\textit{Rev. Union Mat. Argent.} \textbf{48} (2007), 57--71.

\bibitem{AFGaV}
Andruskiewitsch, N.; Fantino, F.; Garc\'ia, G. A.; Vendramin, L. 
On Nichols algebras associated to simple racks. 
\textit{Contemp. Math.} \textbf{537} (2011), 31--56.

\bibitem{AGM}
Andruskiewitsch, N.; Galindo, C.; M\"uller, M. 
Examples of finite-dimensional Hopf algebras
with the dual Chevalley property. \textit{Publ. Mat.}, \textbf{61}(2) (2017), 445--474.

\bibitem{AG}
Andruskiewitsch, N.; Gra\~na, M. 
From racks to pointed Hopf algebras. 
\textit{Adv. Math.} \textbf{178} (2003), no. 2, 177--243.

\bibitem{AHS}
Andruskiewitsch, N.; Heckenberger, I.; Schneider, H.-J.
The Nichols algebra of a semisimple Yetter-Drinfeld module. 
\textit{Amer. J. Math.} \textbf{132} 6 (2010), 1493--1547.

\bibitem{AV}
 Andruskiewitsch, N.; Vay, C. 
 Finite dimensional Hopf algebras over the dual group algebra of the symmetric group in three letters. 
 \textit{Comm. Algebra}, \textbf{39}(12) (2011), 4507--4517.

\bibitem{An} Angiono, I. 
A presentation by generators and relations of Nichols algebras of diagonal type and convex orders on root systems. 
\textit{J. Europ. Math. Soc.} \textbf{17} (2015), 2643--2671.

\bibitem{BGGM}
Bagio, D.; Garc\'ia, G. A.; Giraldi, J.M.J.; M\'arquez, O. 
Finite-dimensional Nichols algebras over dual Radford
algebras. \textit{J. Algebra Appl.} \textbf{20} (1) (2021), Paper No. 2140001, 39 pp.

\bibitem{FG}
Fantino, F.; Garc\'ia, G. A. 
On pointed Hopf algebras over dihedral groups. 
\textit{Pacific J. Math.} \textbf{252} (2011), no. 1, 69--91.

\bibitem{GGI}
Garc\'ia G. A.; Garc\'ia Iglesias, A.  
Finite dimensional pointed Hopf algebras over $\mathbb{S}_{4}$. 
\textit{Israel J. Math.} \textbf{183} (2011), 417--444.

\bibitem{GJG}
Garc\'ia, G. A.; Giraldi, J. M. J.
On Hopf algebras over quantum subgroups. 
\textit{J. Pure Appl. Algebra} \textbf{223} (2019) 738--768.

\bibitem{GIV}
Garc\'ia Iglesias, A.; Vay, C. 
Finite-dimensional pointed or copointed Hopf algebras over affine racks. 
\textit{J. Algebra} \textbf{397} (2014), 379--406.

\bibitem{G} Gra\~na, M. 
On Nichols algebras of low dimension. 
New trends in Hopf algebra theory (La Falda, 1999), 111--134, 
\textit{Contemp. Math.}, \textbf{267}, Amer. Math. Soc., Providence, RI, 2000.

\bibitem{He} Heckenberger, I. 
Classification of arithmetic root systems. 
\textit{Adv. Math.} \textbf{220} (2009), no. 1, 59--124.


\bibitem[HMV]{HMV} Heckenberger, I.; Mehir, E.; Vendramin, L.;
Simple Yetter-Drinfeld modules over groups with prime dimension and 
a finite-dimensional Nichols algebra.
Preprint: \texttt{https://arxiv.org/abs/2306.02989}.

\bibitem{HS} Heckenberger, I.; Schneider, H-J.
Hopf algebras and root systems. 
\textit{Mathematical Surveys and Monographs}, \textbf{247}. American Mathematical Society, Providence, RI, 2020. xix+582 pp.

\bibitem{K} Kac, G.I.
Extensions of groups to ring groups.
\textit{Mat. Sb.} \textbf{5} (1968), 451--474.

\bibitem{M} Mastnak, M.
Hopf algebra extensions arising from semi-direct product of groups.
\textit{J. Algebra} \textbf{251} (2002), 413--434.

\bibitem{Ma} Masuoka, A.
Calculations of some groups of Hopf algebra extensions.
\textit{J. Algebra} \textbf{191} (1997), 568--588.

\bibitem{NZ}
Nichols, W. D.; Zoeller, M. B. 
A Hopf algebra freeness theorem. \textit{Amer. J. Math.} 
\textbf{111} (1989), no. 2, 381--385.

\bibitem{R}  Radford D. E.
Hopf algebras,
\textit{Series on Knots and Everything}  {\bf 49},  World Scientific Publishing Co.\ Pte.\ Ltd., Hackensack, NJ, 2012.

\bibitem{R2}  Radford D. E.
 On oriented quantum algebras derived from representations 
 of the quantum double of a finite-dimensional Hopf algebra.
 \textit{J. Algebra} \textbf{270} (2003), no. 2, 670--695.
   
\bibitem{Serre} Serre J. P.
 Linear representations of finite groups. 
 Translated from the second French edition by Leonard L. Scott. 
 \textit{Graduate Texts in Mathematics}, Vol. \textbf{42}. Springer-Verlag, New York-Heidelberg, 1977. {\rm x}+170 pp.

 \bibitem{Shi} Shi, Y. 
 Finite-dimensional Hopf algebras over the Kac-Paljutkin algebra $H_8$.
\textit{Rev. Un. Mat. Argentina} \textbf{60} (2019), no. 1, 265--298. 

\bibitem{V} Vendramin, L. 
 Nichols algebras associated to the transpositions of the symmetric group are twist-equivalent. 
 \textit{Proc. Amer. Math. Soc.} \textbf{140} (2012), no. 11, 3715--3723.

\bibitem{X1} Xiong, R. 
On Hopf algebras over the unique 12-dimensional Hopf algebra without the dual Chevalley property. 
\textit{Comm. Algebra} \textbf{47} (2019), no. 4, 1516--1540.  


\bibitem{XH} Xiong, R.; Hu, N. 
Classification of finite-dimensional Hopf algebras over dual Radford algebras. 
\textit{Bull. Belg. Math. Soc. Simon Stevin} \textbf{28} (2022), no. 5, 633--688.

\bibitem{Z1}
Zheng, Y.; Gao, Y.; Hu, N.; Shi, Y. 
On some classification of finite-dimensional Hopf algebras over the Hopf algebra $H_{b\: 1}^*$ Kashina. \textit{Comm. Algebra} \textbf{51} (2023), no. 1, 350--371. 

\bibitem{Z2}
Zheng, Y.; Gao, Y.; Hu, N. 
Finite-dimensional Hopf algebras over the Hopf algebra $H_{b:1}$ of Kashina. \textit{J. Algebra} \textbf{567} (2021), 613--659. 

\bibitem{Z3}
Zheng, Y.; Gao, Y.; Hu, N. 
Finite-dimensional Hopf algebras over the Hopf algebra $H_{d:-1,1}$ of Kashina. \textit{J. Pure Appl. Algebra} \textbf{225} (2021), no. 4, Paper No. 106527, 37 pp.


\bibitem{Z} Zhu, Y. 
Hopf algebras of prime dimension, \textit{Int. Math. Res. Not.} \textbf{1} (1994), 53--59.
 
\end{thebibliography}
\end{document}